\documentclass[11pt]{amsart}

\usepackage{amsmath}
\usepackage{fullpage}
\usepackage{xspace}
\usepackage[psamsfonts]{amssymb}
\usepackage[latin1]{inputenc}
\usepackage{graphicx,color}
\usepackage[curve]{xypic} 
\usepackage{hyperref}
\usepackage{graphicx}

\usepackage{amsmath}%
\usepackage{amsthm}%
\usepackage{amscd}
\usepackage{amsfonts}%
\usepackage{amssymb}%
\usepackage{graphicx}

\usepackage{mathrsfs}

\usepackage{tikz}
\usetikzlibrary{matrix,arrows}

\usepackage{tikz-cd}

\newtheorem{theorem}{Theorem}[section]
\theoremstyle{plain}

\newtheorem{conjecture}[theorem]{Conjecture}
\newtheorem{corollary}[theorem]{Corollary}
\newtheorem{lemma}[theorem]{Lemma}
\newtheorem{proposition}[theorem]{Proposition}
\newtheorem{question}[theorem]{Question}

\theoremstyle{remark}

\newtheorem{example}[theorem]{Example}

\numberwithin{equation}{section}

\newcommand{\Mfrak}{\mathfrak{M}}
\newcommand{\Nfrak}{\mathfrak{N}}

\newcommand{\pfrak}{\mathfrak{p}}

\newcommand{\Qfrak}{\mathfrak{Q}}

\newcommand{\Acal}{\mathscr{A}}

\newcommand{\Dcal}{\mathscr{D}}

\newcommand{\Kcal}{\mathscr{K}}
\newcommand{\Lcal}{\mathscr{L}}
\newcommand{\Mcal}{\mathscr{M}}

\newcommand{\Rcal}{\mathscr{R}}
\newcommand{\Scal}{\mathscr{S}}

\newcommand{\Gcal}{\mathscr{G}}

\newcommand{\Pro}{\mathbb{P}}

\newcommand{\Z}{\mathbb{Z}}
\newcommand{\C}{\mathbb{C}}

\newcommand{\F}{\mathbb{F}}
\newcommand{\Q}{\mathbb{Q}}
\newcommand{\R}{\mathbb{R}}
\newcommand{\N}{\mathbb{N}}
\newcommand{\A}{\mathbb{A}}

\newcommand{\ord}{\mathrm{ord}}

\newcommand{\rk}{\mathrm{rank}}
\newcommand{\supp}{\mathrm{supp}\,}

\newcommand{\Id}{\mathrm{Id}}
\newcommand{\im}{\mathrm{im}}
\newcommand{\dom}{\mathrm{dom}}

\newcommand{\Rat}{\mathrm{Rat}}

\newcommand{\Pole}{\mathrm{Pole}}

\newcommand{\Symm}{\mathrm{Symm}}

  \DeclareFontFamily{U}{wncy}{}
    \DeclareFontShape{U}{wncy}{m}{n}{<->wncyr10}{}
    \DeclareSymbolFont{mcy}{U}{wncy}{m}{n}
    \DeclareMathSymbol{\Sha}{\mathord}{mcy}{"58}

\begin{document}
\title{Notes on the DPRM property for listable structures}

\author{Hector Pasten}
\address{Departamento de Matem\'aticas,
Pontificia Universidad Cat\'olica de Chile.
Facultad de Matem\'aticas,
4860 Av.\ Vicu\~na Mackenna,
Macul, RM, Chile}
\email[H. Pasten]{hpasten@gmail.com}%

\thanks{Supported by ANID (ex CONICYT) FONDECYT Regular grant 1190442 from Chile. }

\date{\today}

\begin{abstract} A celebrated result by M. Davis, H. Putnam, J. Robinson, and Y. Matiyasevich shows that a set of integers is listable if and only if it is positive existentially definable in the language of arithmetic. We investigate analogues of this result over structures endowed with a listable presentation. When such an analogue holds, the structure is said to have the DPRM property.  We prove several results addressing foundational aspects around this problem, such as uniqueness of the listable presentation, transference of the DPRM property under interpretation, and its relation with positive existential bi-interpretability. A first application of our results is the rigorous proof of (strong versions of) several folklore facts regarding transference of the DPRM property. Another application of the theory we develop is that it will allow us to link various Diophantine conjectures to the question of whether the DPRM property holds for global fields. This last topic includes a study of the number of existential quantifiers needed to define a Diophantine set.
\end{abstract}

\maketitle

\setcounter{tocdepth}{1}


\section{Introduction}

\subsection{The DPRM theorem} We write $\N$ for the semi-ring of non-negative integers. A set $X\subseteq \N^r$ is \emph{listable} (aka. computably or recursively enumerable) if its elements can be listed by a Turing machine. On the other hand, $X\subseteq \N^r$ is \emph{Diophantine} if there is a polynomial $P\in \Z[x_1,...,x_r,y_1,...,y_k]$ (depending on $X$) satisfying $X=\{{\bf a}\in \N^r : \exists {\bf b}\in \N^k, P({\bf a},{\bf b})=0\}$. An elementary fact is that $X\subseteq \N^r$ is Diophantine if and only if it is positive existentially definable over $\N$ in the language of arithmetic $\Lcal_a=\{0,1,+,\times, =\}$.  Also, it is a classical remark that every Diophantine set is listable. A celebrated result by Davis-Putnam-Robinson \cite{DPR} and Matiyasevich \cite{Matiyasevich} gives the converse.
\begin{theorem}[DPRM theorem] A set $X\subseteq \N^r$ is Diophantine if and only if it is listable. 
\end{theorem}
An immediate consequence of this result is a negative solution to Hilbert's tenth problem. However, the DPRM theorem goes far beyond undecidability; it gives a complete and satisfactory classification of the positive existentially definable sets of the structure $\N$ over $\Lcal_a$. 

In this article we investigate extensions of the DPRM theorem in the setting of \emph{listable structures} (cf. Section \ref{SecIntroListable}). When such a structure satisfies an analogue of the DPRM theorem, we will say that it has the \emph{DPRM property} (cf. Section \ref{SecIntroDPRM}). We prove a number of results addressing foundational aspects of the relation between listable sets and positive existentially definable sets in listable structures; see Section \ref{SecIntroDPRM} for a brief summary of our results on this setting. 

Finally, we discuss applications of our results in two different directions. First, we will provide rigorous proofs for several folklore facts concerning transference of the DPRM property, in a strong form (cf. Section \ref{SecIntroAppFolklore}). And secondly, we will apply our results to show that several number-theoretical conjectures are in fact closely related to the question of whether global fields have the DPRM property or not (cf. Section \ref{SecIntroAppGlobal}).

\subsection{Analogues of the DPRM theorem.} \label{SecIntroAnalogues} Given a first order language $\Lcal$, an $\Lcal$-structure $\Mfrak=(M;\Lcal)$  is \emph{recursive} if there is a surjective map onto the domain $\theta:\N\to M$ satisfying that the pull-back of the interpretation of each element of the signature $s\in \Lcal$ is decidable; cf. \cite{FroShe}. 

Starting with Denef's work \cite{DenefZT} on $\Z[T]$, analogues of the DPRM over other structures have been investigated in the context of \emph{recursive rings}. We refer the reader to Section \ref{SecDPRMknown} for a summary of the known results on analogues of the DPRM theorem.

However, the setting of recursive rings is too restrictive. For instance, the crucial result by Davis-Putnam-Robinson \cite{DPR} does not concern rings as the signature contains an exponential function. Furthermore, the condition that the structure under consideration be recursive is unnatural in the study of extensions of the DPRM theorem. If $\Mfrak$ is a recursive  and we expand its signature  by a positive existentially definable relation, the new structure can fail to be recursive. For instance, consider a listable undecidable set $H\subseteq \N$ (which is Diophantine by DPRM) and note that the structure $(\N;0,1,+,\times, H,=)$ is not recursive, although its positive existentially definable sets are  the same as those of $(\N;0,1,+,\times,=)$. This problem is avoided by considering \emph{listable structures}. 


\subsection{Listable structures} \label{SecIntroListable} Let $\Lcal$ be a first order language and let $\Mfrak=(M;\Lcal)$ be an $\Lcal$-structure. A \emph{listable presentation} of $\Mfrak$ is a surjective map $\rho:\N\to M$ such that for every $s\in \Lcal$ the pull-back under $\rho$ of the interpretation of $s$ is a listable set. In this way, $\rho$ affords a notion of listable set on $\Mfrak$ with respect to $\rho$:  A set $X\subseteq M^r$ is \emph{$\rho$-listable} if its pull-back under $\rho$ is listable.

If  $\Mfrak$ admits a listable presentation, we say that it is a \emph{listable structure}, aka. positive  or  recursively enumerable structure, see \cite{Selivanov1, Selivanov2}. In Section \ref{SecListable} we give a detailed study of listable presentations tailored to our intended applications on extensions of the DPRM property. Among other results, we prove a transference criterion (Proposition \ref{PropIntimplList}), a characterization of $\rho$-listable sets of $\Mfrak$ (Lemma \ref{LemmaCharList}), and the existence of universal $\rho$-listable sets (cf. Lemma \ref{LemmaUnivList} together with Corollary \ref{CoroBij}). We give rigorous proofs of all these results ---our arguments do not rely on a naive notion of ``real-world algorithm''. These statements seem to be well-known to the experts, but we were unable to find proofs in the literature.

In Section \ref{SecEquivListStr} we discuss a notion of equivalence of listable presentations. A central theme in our analysis is whether all the listable presentations of a given listable structure are equivalent to each other; i.e. the problem of \emph{unique listability} for a structure. Our Theorem \ref{ThmUL} provides a very general criterion for unique listability, which  implies unique listability of finitely generated structures (Proposition \ref{PropFG}). Uniquely listable structures are convenient since they have an intrinsic notion of listable sets. We show that this last feature essentially characterizes unique listability (Theorem \ref{ThmEquivCompare}). We remark that listable structures are not our main goal, and we cover them because they are necessary in our study of the DPRM property.


\subsection{The DPRM property} \label{SecIntroDPRM} All first order definitions are understood to be   \emph{without parameters}, unless explicitly stated otherwise. Let $\Mfrak$ be a listable $\Lcal$-structure. It is easy to show that if a set $X\subseteq M^r$ is positive existentially $\Lcal$-definable, then it is $\rho$-listable for every listable presentation $\rho$ of $\Mfrak$ (Corollary \ref{CoroTotList}).  A listable $\Lcal$-structure $\Mfrak$ has the \emph{DPRM property} when the converse holds, namely, if the positive existentially $\Lcal$-definable sets over $\Mfrak$ are the same as those which are $\rho$-listable for every listable presentation $\rho$ of $\Mfrak$ ---naturally, the definition simplifies when $\Mfrak$ is uniquely listable. Our main results on the DPRM property concern the following aspects.

\subsubsection{The number of existential quantifiers} Let $\Mfrak$ be an $\Lcal$-structure with domain $M$. Given a set $X\subseteq M^r$ we define $\rk^{p.e.}_\Mfrak(X)$ as the least number of existential quantifiers needed to give a positive existential $\Lcal$-definition of $X$ with parameters from $M$. For instance, if $X$ is a singleton then $\rk^{p.e.}_\Mfrak(X)=0$. The quantity $\rk^{p.r.}_\Mfrak(X)$ will be called \emph{positive existential rank} of $X$. Theorem \ref{ThmCatDPRM} shows that if a uniquely listable structure $\Mfrak=(M;\Lcal)$ satisfies the DPRM property, then (under some mild assumptions) for each $r\ge 1$ there is a uniform parametric positive existential definition of all the positive existentially definable sets of $\Mfrak$ contained in $M^r$. 

As a consequence, in this setting we have uniform boundedness of the positive existential rank: There is a bound $B_\Mfrak(r)$ depending only on $\Mfrak$ and $r$ such that $\rk^{p.e.}_\Mfrak(X)\le B_\Mfrak(r)$ for every positive existentially $\Lcal$-definable set $X\subseteq M^r$.

Such a uniform boundedness property is remarkable and it can fail even in some very natural structures. In fact, building on work by Koll\'ar \cite{Kollar}, in Theorem \ref{ThmCt} we prove unboundedness of the positive existential rank for the field of complex rational functions $\C(t)$. With a similar argument, the same result can be obtained over any uncountable large field of characteristic zero (e.g. $\R$ or $\Q_p$) but we only state the results over $\C$ for the sake of simplicity.

\subsubsection{Transference of the DPRM property} Denef's proof \cite{DenefZT} of the DPRM property for the ring $\Z[T]$ implicitly developed a transference principle that takes as input the DPRM property in the semi-ring $\N$ in order to deduce the DPRM property for a recursive integral domain. See for instance \cite{ZahidiOKt} where Denef's transference principle is made explicit, and see the detailed discussion in \cite{DemeyerThesis}. 

In Theorem \ref{ThmTransferDPRM} we prove a much more general transference result. Consider languages $\Lcal_1,\Lcal_2$ and uniquely listable structures $\Mfrak_1, \Mfrak_2$ over these languages respectively, such that $\Mfrak_2$ has the DPRM property and each structure admits a positive existential interpretation in the other. Theorem \ref{ThmTransferDPRM} shows that $\Mfrak_1$ inherits the DPRM property if and only if the graph of the self-interpretation of $\Mfrak_1$ obtained by composing the two given interpretations, is positive existentially definable over $\Mfrak_1$. In the special case when $\Mfrak_2$ is the semi-ring $\N$ and $\Mfrak_1$ is a recursive integral domain, this specializes to Denef's transference principle (see Corollary \ref{CoroDPRM}).

Furthemore, Theorem \ref{ThmTransferDPRM} shows that under the previous assumptions, $\Mfrak_2$ has the DPRM property if and only if $\Mfrak_1$ and $\Mfrak_2$ are \emph{positive existentially bi-interpretable} in the sense introduced in Section \ref{SecHomotopy}. This aspect is not covered by the classical version of Denef's transference principle.

Let us stress the fact that for two structures $\Mfrak_1$ and $\Mfrak_2$ to be positive existentially bi-interpretable it is not enough that each one admits a positive existential interpretation in the other. By definition, we moreover require that the self-interpretations of $\Mfrak_1$ and $\Mfrak_2$ obtained by composing the two interpretations be \emph{positive existentially homotopic} to the identity interpretation, see Section \ref{SecHomotopy} for details. If we drop the condition that all formulas involved in the discussion be positive existential, then we are back in the classical setting of homotopy of interpretations introduced in \cite{Ambos}. Several useful fundamental facts about bi-interpretability in this sense have been  recently proved in \cite{AKNS}.  The necessary material for the positive existential counterpart is developed in Section \ref{SecHomotopy}.

\subsubsection{Model-theoretical characterization} Our transference results allow us to obtain, under some mild assumptions, a characterization of the uniquely listable structures having the DPRM property.
\begin{theorem}[cf. Theorem \ref{ThmCharDPRM}] Let $\Lcal$ be a fist order language and let $\Mfrak$ be a uniquely listable $\Lcal$-structure with an infinite domain such that the relation $\ne$ is positive existentially $\Lcal$-definable on $\Mfrak$. Then the following are equivalent:
\begin{itemize}
\item[(i)] $\Mfrak$ has the DPRM property. 
\item[(ii)] $\Mfrak$ is positive existentially bi-interpretable with the semi-ring $\N$.
\end{itemize}
\end{theorem}
Thus, uniquely listable structures having the DPRM property are essentially the same as those which are positive existentially bi-interpretable with $\N$. In particular, under the assumptions of the theorem, uniquely listable structures which are not bi-interpretable with $\N$ do not have the DPRM property.

In view of Proposition \ref{PropFG}, unique listability holds for finitely generated structures. In particular, this leaves the problem of classifying which finitely generated structures are positive existentially bi-interpretable with $\N$. Let us remark that the analogous problem for \emph{first order} bi-interpretations with $\N$ is fully solved in the recent work \cite{AKNS} in the case of commutative finitely generated rings.


\subsection{Application: Some folklore transference results} \label{SecIntroAppFolklore} 

In Section \ref{SecExamplesDPRM} we use our general results on the DPRM property to prove various folklore facts for which no complete proof seems to be available in the literature. These include, for instance, the fact that $\Z$ is Diophantine in $\Q$ if and only if $\Q$ has the DPRM property, as well as analogues for rings of integers and function fields. In fact, we give a more precise version of the aforementioned equivalence, which relates these conditions to positive existential bi-interpretability with $\N$.


\subsection{Application: Diophantine conjectures over global fields}\label{SecIntroAppGlobal}

\subsubsection{An algebraicity conjecture} A real number $x\in \R$ will be called \emph{left-Diophantine} if it is the supremum of a Diophantine subset of $\Q$. We will show that every real algebraic number is left-Diophantine, and that the class $\Dcal$ of left-Diophantine numbers is contained in the class $\Lambda$ of real numbers that can be ``described'' by a Turing machine by producing rational approximations from below: the \emph{left-listable numbers} (also known as left recursively enumerable numbers, or left computably enumerable numbers). These lower and upper bounds for $\Dcal$ are far apart, since the class $\Lambda$ contains rather exotic transcendental numbers. We conjecture that, in fact, the set $\Dcal$ is as small as possible: it is just the field of real algebraic numbers. In particular, we conjecture that $\Dcal$ is a field, which should be contrasted with the fact that $\Lambda$ is not a field \cite{Ambos}.

As we will explain,  the existence of at least \emph{one} left-listable real number which is not left-Diophantine would be enough to imply that $\Z$ is not Diophantine in $\Q$.

See Section \ref{SecLeft} for  details. In particular, we show that the conjecture that $\Dcal$ is exactly the field of real algebraic numbers follows from a version of Mazur's conjecture on the topology of rational points \cite{MazurConj1,MazurConj2,CTSSD} (the necessary material on Mazur's conjecture is recalled in Section \ref{SecMazur}).


\subsubsection{The number of existential quantifiers for global fields} Our result on unboundedness of the positive existential (p.e.) rank for the field $\C(t)$ (Theorem \ref{ThmCt}) leads us to conjecture that the same failure of uniform boundedness might hold for global fields such as $\Q$ and $k(t)$ with $k$ a finite field. As we will show in Proposition \ref{PropBddnonDioph}, the latter would imply that $\Z$ and $k[t]$ are not Diophantine in $\Q$ and $k(t)$ respectively. Furthermore, Proposition \ref{PropBddnonDioph} shows that unboundedness of the p.e. rank over a global field $K$ would imply that $K$ is not positive existentially bi-interpretable with $\N$.

After a first version of this work was released, Philip Dittmann informed us about his joint work with Nicolas Daans and Arno Fehm \cite{DDF} where they carry out an independent and very detailed study the minimal number of existential quantifiers required to define Diophantine sets over fields. In particular, they relate this notion to other measures of complexity of Diophantine sets, and they also point out the link with the question of whether $\Z$ is Diophantine in $\Q$. However,  our result on $\C(t)$ and the connection of the p.e. rank with the DPRM property are not covered in \cite{DDF}.

\subsubsection{A Diophantine approximation conjecture}
In Section \ref{SecKey} we introduce a conjecture in Diophantine approximation (Conjecture \ref{ConjKey}) which would imply that the analogue of the DPRM theorem fails over every global field. 

In simple terms, the main idea of the conjecture is the following: If $K$ is a global field, $v$ is a place of $K$, and $X$ is a variety over $K$ whose $K$-rational points are Zariski dense, then we expect that there is an effective divisor $D$ on $X$ defined over $K$  such that some sequence of $K$-rational points on $X$ $v$-adically approaches $D$. The precise statement of the conjecture is more technical because  it allows one to choose $D$ in a family of divisors, and it allows one to discard families of $K$-rational points coming from lower dimensional varieties, as such points might have anomalous Diophantine approximation properties. We provide some evidence for this conjecture in Section \ref{SecKey}; especially, we show that in the number field case it would follow from general conjectures of Mazur on the topology of rational points.  In Section \ref{SecnD} we show that our Diophantine approximation conjecture would imply that various  natural subsets of global fields are not Diophantine, such as $\Z$  in $\Q$, as well as $\F_p[t]$ and $\{t^n: n\in \N\}$ in $\F_p(t)$ ---this last case requires results on functional transcendence.



\section{Notation and basic facts}\label{SecPrelim}

\subsection{Functions and sets} Given sets $X,Y$ and a function $f:X\to Y$, we define 
$$
\Gamma(f)=\{(y,x)\in Y\times X : y=f(x)\}\subseteq Y\times X
$$

For a positive integer $r$ we let $f^{(r)}:X^r\to Y^r$ be the map $(x_1,...,x_r)\mapsto (f(x_1),...,f(x_r))$. Given $S\subseteq Y^r$ we let $f^*(S)=(f^{(r)})^{-1}(S)\subseteq X^r$. Although the notation $f^*$ does not refer to $r$, the value of $r$ will always be clear from the context. We will use bold fonts to indicate a tuple whenever the number of coordinates is clear from the context, for instance, ${\bf a}=(a_1,...,a_r)$.

\subsection{Recursive functions} 

A partial function over $\N$ is a function $f:A\to \N$ where $A\subseteq \N^r$ for some $r$ called arity. The set of \emph{recursive functions} $\Rcal$ is the smallest class of partial functions over $\N$ satisfying the next two conditions:
\begin{itemize}
\item It contains the successor function $S(x)=x+1$, all coordinate projections, and the constant function $0$ of each arity $r\ge 1$. 
\item It is closed under composition, recursion, and the minimalization operator $\mu$.
\end{itemize}
For details, see (for instance) \cite{CoriLascar} or any other textbook on the subject.

It is a classical result that $\Rcal$ is precisely the class of partial functions that can be computed by Turing machines ---the domain corresponds to those inputs where the corresponding machine halts.

More generally, we say that a function $f:A\to \N^k$ for some $A\subseteq \N^r$ is \emph{recursive} if $f=(f_1,...,f_k)$ where each $f_j:A\to \N$ is recursive in the previous sense. For $k=1$ the definitions agree, of course.

A recursive function of arity $r$ is \emph{total} if its domain is $\N^r$.

\subsection{Decidable and listable sets}

Let $X\subseteq \N^r$ be a set. We say that $X$ is \emph{decidable} if its characteristic function $\chi_X:\N^r\to \{0,1\}\subseteq \N$ is total recursive. We say that $X$ is \emph{listable} if it is the domain of a recursive function. We have the following standard characterizations:
\begin{itemize}
\item $X$ is decidable if and only if $X$ and its complement $X^c$ are both listable.
\item $X$ is listable if and only if  it is either empty or the image of a total recursive function $f:\N\to \N^r$. Furthermore, if $X$ is infinite, then one can ask $f$ to be injective.
\end{itemize}
The next two lemmas are straightforward.
\begin{lemma}\label{LemmaListImage} Let $f:\N^a\to \N^b$ be a total recursive function and let $B\subseteq \im(f)\subseteq \N^b$. We have that $B$ is listable if and only if $f^{-1}(B)$ is listable. 
\end{lemma}
\begin{lemma}[Basic operations with listable sets]\label{LemmaBooleanN} The class of listable sets is closed under finite unions, finite intersections, permutation of coordinates, coordinate projections, and image and preimage under recursive functions.
\end{lemma}
We also recall the following important result together with two remarkable consequences.
\begin{theorem}[Universal recursive function; the enumeration theorem, cf. \cite{CoriLascar}] Let $r\ge 1$. There is a partial recursive function $\phi^{univ}_r:U_r\to \N$ with domain $U_r\subseteq \N^{r+1}$ such that for every partial recursive function $f$ of arity $r$, there is an integer $i(f)\in \N$ such that $\phi^{univ}_r(i(f),x_1,...,x_r)=f(x_1,...,x_r)$ as partial functions.
\end{theorem}

\begin{corollary}[A universal listable set]\label{CoroUr} Let $r\ge 1$. The set $U_r=\dom(\phi^{univ}_r)\subseteq \N^{r+1}$ is listable and it has the following property: For every listable set $X\subseteq \N^r$ there is $n_X\in \N$ such that 
$X= \{{\bf a}\in \N^r : (n_X,{\bf a})\in U_r\}$.
\end{corollary}

\begin{corollary}[An undecidable listable set: The halting problem] Define the partial recursive function $h(x)=\phi^{univ}_1(x,x)$. The set $H=\dom(h)\subseteq \N$ is listable but it is not decidable.
\end{corollary}
\subsection{Structures} Let $\Lcal$ be a first order language consisting of symbols of constants, relations, and functions. For an $\Lcal$-structure $\Mfrak$, its domain is denoted by $M=|\Mfrak|$ and for each $s\in \Lcal$ we let $s^\Mfrak$ be the interpretation of $s$. Thus: 
\begin{itemize}
\item If $s$ is a symbol of a constant, then $s^\Mfrak\in M$.
\item If $s$ is a symbol of an $n$-ary relation, then $s^\Mfrak\subseteq M^n$. 
\item If $s$ is a symbol of an $n$-ary function, then $s^\Mfrak\subseteq M^{n+1}$ (the graph).
\end{itemize}

We make the important assumption that \emph{we only consider languages containing the binary relation symbol `` $=$'', and we only consider structures where this symbol is interpreted as equality}. We refer to this assumption as \emph{equality hypothesis}.

For an $\Lcal$-formula $\phi$, the notation $\phi[x_1,...,x_n]$ means that the free variables of $\phi$ are among the variables $x_1,...,x_n$, and all these variables are free or do not occur in $\phi$. We do not allow parameters unless explicitly stated otherwise. Given an $\Lcal$-formula $\phi[x_1,...,x_n]$, the interpretation of it over $\Mfrak$ is
$
\phi^\Mfrak = \{{\bf a}\in M^n : \Mfrak\models \phi[{\bf a}]\}\subseteq M^n
$
where $\phi[{\bf a}]$ means that $\phi[x_1,...,x_n]$ is interpreted in $\Mfrak$ by interpreting $(x_1,...,x_n)$ as ${\bf a}\in M^n$.

An $\Lcal$-formula $\phi$ is \emph{positive existential} (denoted as p.e.) if the only quantifier it uses is $\exists$, it does not use negations, and the only connectives that it uses are $\vee$ and $\wedge$.

\subsection{Interpretations}  Let $\Lcal,\Kcal$ be languages, and let $\Mfrak,\Nfrak$ be structures over these languages with domains $M,N$ respectively. Let $r\ge 1$. An \emph{interpretation of $\Nfrak$ in $\Mfrak$ of rank $r$} is a map $\theta:X\to N$ where $X=\dom(\theta)\subseteq M^r$, satisfying the following properties:
\begin{itemize}
\item[(Int1)] $\theta:\dom(\theta)\to N$ is surjective onto $N$.
\item[(Int2)] $\dom(\theta)$ is $\Lcal$-definable over $\Mfrak$.
\item[(Int3)] For each $s\in \Kcal$, the set $\theta^*(s^\Nfrak)$ is $\Lcal$-definable over $\Mfrak$.
\end{itemize}
We use the notation $\theta:\Mfrak\dasharrow \Nfrak$ to indicate that $\theta$ is an interpretation of $\Nfrak$ in $\Mfrak$, and the rank $r$ is denoted by $\rk(\theta)$. We say that the interpretation $\theta:\Mfrak\dasharrow \Nfrak$ is \emph{positive existential} (p.e.) if $\dom(\theta)$ and each  $\theta^*(s^\Nfrak)$ for $s\in \Kcal$ are p.e. $\Lcal$-definable.

\begin{lemma}[Pull-back of definable sets under interpretations] Let $\Lcal$ and $\Kcal$ be languages. Consider an $\Lcal$-structure $\Mfrak$, a $\Kcal$-structure $\Nfrak$, and an interpretation $\theta:\Mfrak\dasharrow \Nfrak$ of rank $r$. Let $S\subseteq N^m$. If $S$ is $\Kcal$-definable, then $\theta^*(S)\subseteq M^{rm}$ is $\Lcal$-definable. Furthermore, if the interpretation is p.e. and the set $S$ is p.e. $\Kcal$-definable, then $\theta^*(S)$ is p.e. $\Lcal$-definable. 
\end{lemma}
A useful standard fact is that interpretations can be composed. 
\begin{lemma}[Composition of interpretations]\label{LemmaCompInt} For $i=1,2,3$, let $\Lcal_i$ be a language and $\Mfrak_i$ an $\Lcal_i$-structure with domain $M_i$. Let $\theta_1:\Mfrak_1\dasharrow \Mfrak_2$ and $\theta_2:\Mfrak_2\dasharrow \Mfrak_3$ be interpretations. There is an interpretation $\zeta:\Mfrak_1\dasharrow \Mfrak_3$ with the following properties:
\begin{itemize}
\item[(i)]  $\rk (\zeta)=\rk (\theta_1)\cdot \rk(\theta_2)$

\item[(ii)] $\dom(\zeta)=\theta_1^*(\dom(\theta_2))\subseteq M_1^{\rk(\zeta)}$.

\item[(iii)] $\zeta: \dom(\zeta)\to M_3$ is defined by $\zeta=\theta_2\circ \theta_1^{(\rk(\theta_2))}$. 
\item[(iv)] If $\theta_1$ and $\theta_2$ are p.e., then $\zeta$ is p.e.
\end{itemize}
\end{lemma}
In the situation of the previous lemma, the interpretation $\zeta$ is said to be the \emph{composition of the interpretations $\theta_1$ and $\theta_2$}, which we denote by $\zeta=\theta_2\bullet \theta_1$. One directly checks

\begin{lemma} \label{LemmaCat} Structures over languages together with p.e. interpretations form a category.
\end{lemma}

\subsection{Homotopy of interpretations}\label{SecHomotopy} For $i=1,2$ let $\Mfrak_i$ be an $\Lcal_i$-structure with domain $M_i$. Let $\theta,\theta':\Mfrak_1\dasharrow \Mfrak_2$ be interpretations of ranks $r$ and $r'$ respectively. We define
$$
K(\theta,\theta')=\{({\bf u},{\bf v})\in M^{r+r'} : {\bf u}\in \dom(\theta),{\bf v}\in \dom(\theta'), \mbox{ and }\theta({\bf u})=\theta'({\bf v})\}\subseteq M^{r+r'}.
$$
Following \cite{AhlbrandtZiegel}, the interpretations $\theta$ and $\theta'$ are said to be \emph{homotopic} if $K(\theta,\theta')$ is $\Lcal_1$-definable over $\Mfrak_1$. See \cite{AhlbrandtZiegel, AKNS} for more details on homotopy of interpretations.

In the special case when $\theta$ and $\theta'$ are p.e. we will need to introduce a refined notion of homotopy. In this case, we say that $\theta$ and $\theta'$ are \emph{positive existential homotopic} if $K(\theta,\theta')$ is p.e. $\Lcal_1$-definable. This will be denoted by $\theta\asymp \theta'$. The results we present below for $\asymp$ are analogous to some of the results given in \cite{AKNS} for homotopy of interpretations (not in the p.e. setting).

\begin{lemma}[P.e. homotopy is an equivalence relation]\label{LemmaHomotopyEquiv} For $i=1,2$, let $\Mfrak_i$ be and $\Lcal_i$-structure. Then $\asymp$ is an equivalence relation in the set of p.e. interpretations of $\Mfrak_2$ in $\Mfrak_1$.
\end{lemma}
\begin{proof} Symmetry is clear. Reflexivity follows from the fact that for any p.e. interpretation $\theta:\Mfrak_1\dasharrow \Mfrak_2$ the set $\theta^*(=)=K(\theta,\theta)$ is p.e. $\Lcal_1$-definable. Let us check transitivity. For $i=1,2,3$ let $\theta_i:\Mfrak_1\dasharrow \Mfrak_2$ be a p.e. interpretation of rank $r_i$, such that $\theta_1\asymp\theta_2$ and $\theta_2\asymp \theta_3$. The set
$
\Omega = \left(K(\theta_1,\theta_2)\times M^{r_3}\right)\cap \left(M^{r_1}\times K(\theta_2,\theta_3)\right)\subseteq M^{r_1+r_2+r_3}
$
is p.e. $\Lcal$-definable, and it equals
$$
\{({\bf x},{\bf y},{\bf z})\in M^{r_1+r_2+r_3} : {\bf x}\in \dom(\theta_1), {\bf y}\in \dom(\theta_2), {\bf z}\in \dom(\theta_3)\mbox{ and }\theta_1({\bf x})=\theta_2({\bf y})=\theta_3({\bf z})\}.
$$
By surjectivity of $\theta_2$ onto the domain of $\Mfrak_2$, we get that $K(\theta_1,\theta_3)$ is the projection of $\Omega$ onto its first $r_1$ and last $r_3$ coordinates. Hence, $K(\theta_1,\theta_3)$ is p.e. $\Lcal_1$-definable.
\end{proof}
\begin{lemma}[Composition respects p.e. homotopy]\label{LemmaHomotopyComp} For $i=1,2,3$ let $\Mfrak_i$ be an $\Lcal_i$-strucutre. Let $\theta_1,\kappa_1:\Mfrak_1\dasharrow \Mfrak_2$  and $\theta_2,\kappa_2:\Mfrak_2\dasharrow \Mfrak_3$ be p.e. interpretations with $\theta_1\asymp\kappa_1$ and $\theta_2\asymp\kappa_2$. Then $\theta_2\bullet\theta_1\asymp\kappa_2\bullet\kappa_1$.
\end{lemma}
\begin{proof} By Lemma \ref{LemmaHomotopyEquiv}, it suffices to prove $\theta_2\bullet\theta_1\asymp\kappa_2\bullet\theta_1$ and $\kappa_2\bullet\theta_1\asymp\kappa_2\bullet\kappa_1$. This can be done as in Lemma 2.1 of \cite{AKNS} since that argument only introduces existential quantifiers (i.e. coordinate projections of definable sets). 
\end{proof}

The (transposed) graph of an interpretation $\theta: \Mfrak\dasharrow\Nfrak$ is 
$$
\Gamma(\theta)=\{(x_0,{\bf x})\in N\times M^{\rk(\theta)}: {\bf x}\in \dom(\theta)\mbox{ and } x_0=\theta({\bf x})\}\subseteq N\times M^{\rk(\theta)}.
$$

For $i=1,2$ let $\Mfrak_i$ be an $\Lcal_i$-structure with domain $M_i$. A \emph{bi-interpretation} is a pair of interpretations $\theta_1:\Mfrak_1\dasharrow \Mfrak_2$ and $\theta_2:\Mfrak_2\dasharrow \Mfrak_1$ such that $\Gamma(\theta_1\bullet \theta_2)$ is $\Lcal_2$-definable over $\Mfrak_2$ and $\Gamma(\theta_2\bullet \theta_1)$ is $\Lcal_1$-definable over $\Mfrak_1$. We say that the bi-interpretation is \emph{positive existential} if $\theta_1$ and $\theta_2$ are p.e. interpretations and both $\Gamma(\theta_1\bullet \theta_2)$ and $\Gamma(\theta_2\bullet \theta_1)$ are p.e. definable. 

\begin{lemma}[Bi-interpretations in terms of homotopy in the p.e. case]\label{LemmaBiIntHomotopy} For $i=1,2$ let $\Mfrak_i$ be an $\Lcal_i$-structure. Let $\theta_1:\Mfrak_1\dasharrow \Mfrak_2$ and $\theta_2:\Mfrak_2\dasharrow \Mfrak_1$ be p.e. interpretations. Then $(\theta_1,\theta_2)$ is a p.e. bi-interpretation if and only if $\theta_2\bullet \theta_1\asymp \Id_{\Mfrak_1}$ and $\theta_1\bullet \theta_2\asymp \Id_{\Mfrak_2}$.
\end{lemma}
\begin{proof} This is because if $\zeta:\Mfrak\dasharrow \Mfrak$ is a p.e. interpretation, then $K(\Id_\Mfrak,\zeta)=\Gamma(\zeta)$.
\end{proof}

When a p.e. bi-interpretation between $\Mfrak_1$ and $\Mfrak_2$ exists, we say that $\Mfrak_1$ and $\Mfrak_2$ are \emph{p.e. bi-interpretable}. One has the following basic property:

\begin{lemma}[Transitivity of p.e. bi-interpretability] \label{LemmaBiIntTransitive} For $i=1,2,3$ let $\Mfrak_i$ be an $\Lcal_i$-structure. Suppose that $\Mfrak_1$ is p.e. bi-interpretable with $\Mfrak_2$ and that $\Mfrak_2$ is p.e. bi-interpretable with $\Mfrak_3$. Then $\Mfrak_1$ and $\Mfrak_3$ are p.e. bi-interpretable. 
\end{lemma}
\begin{proof} Let $\theta_1:\Mfrak_1\dasharrow \Mfrak_2$, $\theta_2:\Mfrak_2\dasharrow \Mfrak_3$, $\kappa_1:\Mfrak_2\dasharrow \Mfrak_1$, and $\kappa_2:\Mfrak_3\dasharrow \Mfrak_2$  be p.e. interpretations such that $(\theta_1,\kappa_1)$ and $(\theta_2,\kappa_2)$ are p.e. bi-interpretations. Using Lemmas \ref{LemmaCat}, \ref{LemmaHomotopyComp}, and \ref{LemmaBiIntHomotopy} we see that $(\kappa_1\bullet\kappa_2)\bullet(\theta_2\bullet\theta_1)=\kappa_1\bullet(\kappa_2\bullet\theta_2)\bullet\theta_1\asymp \kappa_1\bullet\Id_{\Mfrak_2} \bullet\theta_1=\kappa_1 \bullet\theta_1\asymp \Id_{\Mfrak_1}$ and similarly $(\theta_2\bullet\theta_1)\bullet(\kappa_1\bullet\kappa_2)\asymp \Id_{\Mfrak_3}$.  We conclude by Lemma \ref{LemmaBiIntHomotopy}.
\end{proof}

\subsection{Examples}

The language of arithmetic is $\Lcal_a=\{0,1,+,\times,=\}$. By interpreting its symbols in the obvious way, every semiring (such as $\N$) becomes an $\Lcal_a$-structure. The following two simple lemmas are included with just to serve as examples of the previous notions. The proofs are direct applications of Lagrange's $4$-squares theorem.

\begin{lemma}\label{LemmaIntNZ} The $\Lcal_a$-structures $\N$ and $\Z$ are p.e. bi-interpretable.
\end{lemma}
\begin{proof} By Lagrange's $4$-squares theorem the binary relation $\ge$ is p.e. $\Lcal_a$-definable in $\Z$. Consider the maps $\theta_1: \N^2\to \Z$ and $\theta_2:\N=\{n\in \Z: n\ge 0\}\to \N$ given by $\theta_1(a,b)=a-b$ and $\theta_2(n)=n$. One readily checks that they define p.e. interpretations $\theta_1: \N\dasharrow \Z$ and $\theta_2:\Z\to \N$, and that these interpretations give the desired p.e. bi-interpretation.
\end{proof}
\begin{lemma}\label{LemmaQinNZ} Consider $\N$, $\Z$ and $\Q$ as $\Lcal_a$-structures. Then $\Q$ is p.e. interpretable in $\N$ and in  $\Z$. Furthermore, the map $\kappa:\Z\times (\Z-\{0\})\to \Q$ given by $\kappa(a,b)=a/b$ defines a p.e. interpretation of $\Q$ in $\Z$.
\end{lemma}

Bi-interpretation and p.e. bi-interpretation are in fact two very different conditions. For instance, it is worth pointing out that by a celebrated theorem of J. Robinson \cite{JRobinsonQ} $\Z$ is $\Lcal_a$-definable in $\Q$, and it easily follows that $\Z$ and $\Q$ are bi-interpretable as $\Lcal_a$-structures. However, it is not known whether they are p.e. bi-interpretable ---see Section \ref{SecExamplesDPRM} below.


\section{Listable structures}\label{SecListable}
\subsection{Listable presentations} Let $\Mfrak$ be an $\Lcal$-structure with domain $M$. We recall from the introduction that a  \emph{listable presentation} of $\Mfrak$ is a surjective set-theoretical function $\rho:\N\to M$ such that for every $s\in \Lcal$ we have that $\rho^*(s^\Mfrak)$ is a listable set. In this setting, we say that $\Mfrak$ is a \emph{listable $\Lcal$-structure}. Of course, if $\Mfrak$ is a listable $\Lcal$-structure, then $M$ is  countable (possibly finite). We have the basic example:

\begin{lemma}\label{LemmaListN} The $\Lcal_a$-structure $(\N;0,1,+,\times,=)$ is listable, and the identity map $\N\to \N$ is a listable presentation.
\end{lemma}

Let $\rho:\N\to M$ be a listable presentation for $\Mfrak$ and let $X\subseteq M^r$. We say that $X$ is \emph{$\rho$-listable} (resp. \emph{$\rho$-decidable}) if $\rho^*(X)\subseteq \N^r$ is listable (resp. decidable). In particular,  the equality hypothesis implies that the set
$
E_\rho=\{(m,n)\in \N^2 : \rho(m)=\rho(n)\}=\rho^*(=)\subseteq \N^2
$
is listable for every listable presentation $\rho:\N\to M$. We observe that

\begin{lemma}\label{LemmaBijDec} Let $\Mfrak$ be a listable $\Lcal$-structure and let $\rho$ be a listable presentation for it. If $\rho$ is bijective, then $E_\rho$ is the diagonal in $\N^2$ and, in particular, $E_\rho$ is decidable.
\end{lemma}

We say that a set $X\subseteq M^r$ is \emph{totally listable} if for every listable presentation $\rho$ of $\Mfrak$ the set $X$ is $\rho$-listable. Let us remark the following:

\begin{lemma}\label{LemmaNE} Let $\Mfrak$ be a listable $\Lcal$-structure and let $\rho$ be a listable presentation for it. The binary relation $\ne$ on $\Mfrak$ is $\rho$-listable if and only if $E_\rho\subseteq \N^2$ is decidable. In particular, $\ne$ on $\Mfrak$ is totally listable if and only if for every listable presentation $\gamma$ of $\Mfrak$ we have that $E_\gamma$ is decidable.
\end{lemma}
\begin{proof} The set $E_\rho=\rho^*(=)\subseteq \N^2$ is listable. Thus, $E_\rho$ is decidable if and only if $E_\rho^c=\rho^*(\ne)\subseteq \N^2$ is listable. The latter precisely means that $\ne$ on $\Mfrak$ is $\rho$-listable.
\end{proof}
However, it can very well happen that $E_\rho$ is undecidable for a listable presentation. For instance, one can take $R\subseteq \N^2$ a listable equivalence relation which is undecidable (see \cite{Ershov}) and consider the structure $(\N/R, =)$. The quotient map $\pi:\N\to \N/R$ is a listable presentation and $E_\pi =R$ is undecidable. We have an alternative characterization of $\rho$-listable sets over a listable structure.

\begin{lemma}[Characterization of $\rho$-listable sets] \label{LemmaCharList}  Let $\Mfrak$ be a listable $\Lcal$-structure with domain $M$ and let $\rho$ be a listable presentation for $\Mfrak$. Let $X\subseteq M^r$ be a subset. The following are equivalent:
\begin{itemize}
\item[(i)] $X$ is $\rho$-listable; that is, $\rho^*(X)\subseteq \N^r$ is listable.
\item[(ii)]  There is a listable set $Z\subseteq \N^r$ with $\rho^{(r)}(Z)=X$.
\item[(iii)] Either $X$ is empty or there is $f:\N\to \N^r$ total recursive with $\rho^{(r)} (f(\N))=X$.
\end{itemize}
\end{lemma}
\begin{proof} (ii) and (iii) are equivalent by the theory of listable sets over $\N$. 

If (i) holds, then (ii) holds with $Z=\rho^*(X)$.

Assume that (iii) holds with $X$ non-empty. Write $f=(f_1,...,f_r)$. Let $\epsilon=(\epsilon_1,\epsilon_2):\N\to \N^2$ be a total recursive function with image $E_\rho$ (this set is listable). Let $g=(g_0,g_1,...,g_r):\N\to \N^{r+1}$ be a total recursive bijection. The function
$$
\phi(x_1,...,x_r) :=\mu y [f_j(g_0(y))=\epsilon_1(g_j(y))\mbox{ and }\epsilon_2(g_j(y))=x_j\mbox{ for each }j=1,...,r]
$$
is partial recursive. By (iii), the domain of $\phi$ is $$\{(x_1,...,x_r)\in \N^r :  \mbox{there exists $n$ such that } \rho (f_j(n))=\rho(x_j)\}=\rho^*(X).$$ Hence, $\rho^*(X)$ is listable.
\end{proof}
\begin{corollary}[Finite sets are totally listable]\label{CoroFinSets} Let $\Mfrak$ be a listable $\Lcal$-structure with domain $M$ and let $X\subseteq M^r$ be finite. Then $X$ is totally listable. 
\end{corollary}
\begin{proof} Let $\rho$ be a listable presentation for $\Mfrak$. Then $X$ is $\rho$-listable by surjectivity of $\rho:\N\to M$ and item (iii) in Lemma \ref{LemmaCharList}.
\end{proof}
In particular, we get the following consequence which implies that the study of listable sets in structures is more interesting for structures with an infinite domain.
\begin{corollary}\label{CoroFin} Let $\Mfrak$ be a listable $\Lcal$-structure whose domain is finite. Then for every $r$, all subsets of $M^r$ are totally listable.
\end{corollary}

In the special case when a listable structure has infinite domain and it admits a bijective listable presentation, the structure inherits universal listable  sets from $\N$.
\begin{lemma}[Universal $\rho$-listable sets]\label{LemmaUnivList} Let $\Mfrak$ be an $\Lcal$-structure with infinite domain $M$ and let $\rho:\N\to M$ be bijective listable presentation for $\Mfrak$.  Given $r\ge 1$, there is a $\rho$-listable set $U_r(\rho)\subseteq M^{r+1}$ with the following property: For every $\rho$-listable set $X\subseteq M^r$ there is an element $m_X\in M$ such that $X=\{{\bf x}\in M^r : (m_X,{\bf x})\in U_r(\rho)\}$.
\end{lemma}
\begin{proof} Let $U_r\subseteq \N^{r+1}$ be the universal listable set provided by Corollary \ref{CoroUr} and let $U_r(\rho)=\rho^{(r+1)}(U_r)\subseteq M^{r+1}$.
Let $X\subseteq M^r$ be a $\rho$-listable set. Then $\rho^*(X)\subseteq \N^r$ is listable and there is an integer $n_{\rho^*(X)}\in \N$ such that $\rho^*(X)=\{{\bf a}\in \N^r : (n_{\rho^*(X)},{\bf a})\in U_r\}$. Let $m_X=\rho(n_{\rho^*(X)})$. Since $\rho$ is injective, we have $(n_{\rho^*(X)},{\bf a})\in U_r$ if and only if $(m_X,\rho^{(r)}({\bf a}))\in U_r(\rho)$. By surjectivity of $\rho$ we conclude  $X=\rho^{(r)}(\rho^*(X)) = \{{\bf x}\in M^r : (m_X,{\bf x})\in U_r(\rho) \}$.
\end{proof}

The condition that a listable presentation $\rho:\N\to M$ be bijective, will be completely characterized in Corollary \ref{CoroBij} below. 

\subsection{Listability and p.e. definability}

The next lemma is straightforward.

\begin{lemma}\label{LemmaBoolean} Let $\Mfrak$ be a listable $\Lcal$-structure and let $\rho$ be a listable presentation for $\Mfrak$. The class of  $\rho$-listable subsets over $\Mfrak$ is closed under cartesian products, coordinate projections, permutation of coordinates, finite unions, and finite intersections.
\end{lemma}
\begin{corollary}[All p.e. definable sets are totally listable] \label{CoroTotList} Let $\Mfrak$ be a listable $\Lcal$-structure. Every p.e. $\Lcal$-definable set over $\Mfrak$ is totally listable.
\end{corollary}
\begin{proof} Fix a listable presentation $\rho$. Using Lemma \ref{LemmaBoolean} and the fact that $s^\Mfrak$ is $\rho$-listable for each $s\in \Lcal$, one gets that for each atomic formula $\phi$, the set $\phi^\Mfrak$ is $\rho$-listable. Applying Lemma \ref{LemmaBoolean} again we get the result.
\end{proof}
\begin{corollary}\label{CoroNEpeEdec} Let $\Mfrak$ be a listable $\Lcal$-structure. If the binary relation $\ne$ on $\Mfrak$ is p.e. $\Lcal$-definable, then for every listable presentation $\rho$ we have that $E_\rho$ is decidable.
\end{corollary}
\begin{proof} By Lemma \ref{LemmaNE} and Corollary \ref{CoroTotList}.
\end{proof}
Quite often it happens that the inequality $\ne$ is p.e. definable over interesting structures (see for instance \cite{MoretBailly}), so in practice it is not a restrictive assumption to require that $E_\rho$ be decidable for every listable presentation $\rho$ of a structure $\Mfrak$. We will not make this assumption in general.

\subsection{Listability and p.e. interpretations}

\begin{proposition}[Transference of listability] \label{PropIntimplList} For $i=1,2$ let $\Lcal_i$ be a language and $\Mfrak_i$ an $\Lcal_i$-structure with domain $M_i$. Let $\theta:\Mfrak_1\dasharrow \Mfrak_2$ be a p.e. interpretation of rank $r$ and let $\rho$ be a listable presentation for $\Mfrak_1$. Let $X=\rho^*(\dom(\theta))\subseteq \N^r$. There exists a total recursive function $f:\N\to \N^r$  with image $X$, and for any $f$ of this type the function $\gamma=\theta\circ \rho^{(r)}\circ f:\N\to M_2$ is a listable presentation for $\Mfrak_2$ having the following property:  For every $S\subseteq M_2^k$ we have that $S$ is $\gamma$-listable if and only if $\theta^*(S)$ is $\rho$-listable. 
\end{proposition}

\begin{proof} The set $\dom(\theta)$ is p.e. $\Lcal_1$-definable, hence, it is $\rho$-listable by Corollary \ref{CoroTotList}. Thus $X$ is listable and $f$ exists. Take any total recursive $f:\N\to \N^r$ with image $X$. By construction, $\gamma:\N\to M_2$ is surjective. Let $s\in \Lcal$ and let $n$ be such that $s^\Nfrak\subseteq M_2^n$. Then $\theta^*(s^{\Mfrak_2})\subseteq \dom(\theta)^n$ is  p.e. $\Lcal_1$-definable, hence, it is $\rho$-listable by Corollary \ref{CoroTotList}. This implies that $\rho^*\theta^*(s^\Nfrak)\subseteq \N^{rn}$ is listable. Hence, $\gamma^*(s^{\Mfrak_2})=f^*\rho^*\theta^*(s^{\Mfrak_2})\subseteq \N^n$ is listable (pre-image of a listable set under a total recursive function). Finally, let $S\subseteq M_2^k$.  Since $f$ surjects onto $X\subseteq \N^r$ and $\rho^*\theta^*(S)\subseteq X^k\subseteq \N^{rk}$, we have that $\rho^*\theta^*(S)\subseteq \im(f^{(k)})$. We conclude by Lemma \ref{LemmaListImage}: $S$ is $\gamma$-listable if and only if $f^*\rho^*\theta^*(S)=\gamma^*(S)\subseteq \N^k$ is listable, which happens if and only if $\rho^*\theta^*(S)\subseteq \N^{rk}$ is listable. Thus, $S$ is $\gamma$-listable if and only if $\theta^*(S)$ is $\rho$-listable.
\end{proof}
\begin{lemma}\label{LemmaListtoInt} Let $\Mfrak$ be a listable $\Lcal$-structure with domain $M$ and let $\rho$ be a listable presentation for it. Then $\rho:\N\to M$ defines a p.e. interpretation of $\Mfrak$ in the $\Lcal_a$-structure $\N$.
\end{lemma}
\begin{proof}  The map $\rho:\N \to M$ is surjective. For each $s\in \Lcal$ we have that $\rho^*(s^\Mfrak)$ is listable, and by the DPRM theorem $\rho^*(s^\Mfrak)$ is p.e. $\Lcal_a$-definable over $\N$.
\end{proof}

\begin{theorem}\label{ThmListableN} Let $\Mfrak$ be an $\Lcal$-structure. Then $\Mfrak$ is p.e. interpretable in the $\Lcal_a$-structure $\N$ if and only if $\Mfrak$ is listable. 
\end{theorem}
\begin{proof} The forward implication follows from Lemma \ref{LemmaListN} and Proposition \ref{PropIntimplList}. The converse follows from Lemma \ref{LemmaListtoInt}
\end{proof}
\begin{corollary}\label{CoroListableZQ} The $\Lcal_a$-structures $\Z$ and $\Q$ are listable.
\end{corollary}
\begin{proof} By Lemmas \ref{LemmaIntNZ} and \ref{LemmaQinNZ}, and by Theorem \ref{ThmListableN}
\end{proof}
The previous corollary admits, of course, a direct proof.
\subsection{Equivalence of listable presentations} \label{SecEquivListStr}
\begin{lemma} Let $\Mfrak$ be a listable $\Lcal$-structure and let $\rho$ and $\gamma$ be listable presentations for it. If there  is a total recursive function $\phi:\N\to \N$ with $\gamma=\rho\circ \phi$, then there is a total recursive function $\psi:\N\to \N$ with $\rho=\gamma\circ \psi$. 
\end{lemma}
\begin{proof} Let $g=(g_0,g_1,g_2):\N\to \N^3$ be a total recursive function with image $\N\times E_\rho$, which exists since $E_\rho$ is listable. Let us define the function $\psi:\N\to \N$ by
$$
\psi(x)= g_0\left(\mu y[\phi(g_0(y))=g_1(y)\mbox{ and }x=g_2(y)]\right).
$$
Then $\psi:\N\to\N$ is total recursive. Given $x$, the integer $n=\psi(x)$ satisfies $(\phi(n),x)\in E_\rho$, that is, $\gamma(n)=\rho(\phi(n))=\rho(x)$. This proves $\gamma\circ \psi = \rho$.
\end{proof}


If $\rho$ and $\gamma$ are listable presentations of an $\Lcal$-structure $\Mfrak$,  we write $\rho\approx \gamma$ if there is a total recursive function $\phi:\N\to \N$ with $\gamma=\rho\circ \phi$. From the previous lemma it follows that

\begin{proposition} Given a listable $\Lcal$-structure $\Mfrak$, the relation $\approx$ is an equivalence relation on the set of listable presentations of $\Mfrak$.
\end{proposition}

Hence, if $\rho\approx \gamma$ we say that $\rho$ and $\gamma$ are \emph{equivalent}. In the special case that all listable presentations of a structure are equivalent to each other, we say that the structure is \emph{uniquely listable}. 

\begin{proposition}\label{PropFinUL} Let $\Mfrak$ be a listable $\Lcal$-structure. If the domain of $\Mfrak$ is finite, then $\Mfrak$ is uniquely listable.
\end{proposition}
\begin{proof} Let $\rho$ and $\gamma$ be listable presentations of $\Mfrak$. Let $M$ be the domain of $\Mfrak$ and for each $m\in M$ let $A_m=\gamma^{-1}(m)$. Then $\{A_m : m\in M\}$ is a partition of $\N$ and each $A_m$ is listable, hence, decidable (it has listable complement). For each $m\in M$ choose $y_m\in \rho^{-1}(m)$. Define $\phi:\N\to \N$ by $\phi(x)=y_m$ if $x\in A_m$. Then $\phi$ is total recursive (defined by decidable cases) and $\gamma=\rho\circ \phi$. 
\end{proof}

From Lemma \ref{LemmaBooleanN} we deduce
\begin{lemma}[Equivalent listable presentations have the same listable sets] \label{LemmaComparisonListSets} Let $\rho$ and $\gamma$ be listable presentations of an $\Lcal$-structure $\Mfrak$. If $\rho\approx \gamma$, then the class of $\rho$-listable sets is the same as the class of $\gamma$-listable sets over $\Mfrak$. 
\end{lemma}

We also have a partial converse to the previous result. First we need:

\begin{lemma}\label{LemmaBij} Let $\Mfrak$ be a listable $\Lcal$-structure with infinite domain $M$. Let $\rho$ be a listable presentation for $\Mfrak$ and assume that $E_\rho$ is decidable. There is an injective total recursive function $\iota_\rho: \N \to \N$ such that $\rho\circ \iota_\rho : \N \to M$ is bijective.
\end{lemma}
\begin{proof}  Let $\chi_{E_\rho}$ be the characteristic function of $E_\rho$, which is total recursive by assumption. We define a function $h:\N\to \N$ as follows: We set $h(0)=0$ and for $x>0$ we let
$$
h(x)=\mu y \left[\sum_{j<x}\chi_{E_\rho}(h(j),y)=0\right].
$$
The function $h$ is total because $M$ is infinite, and it is recursive because it is defined by minimalization and course-of-values recursion. Given $x>0$, we note that $h(x)$ is the least value of $y$ for which $\rho(y)\ne \rho(h(j))$ for each $j<x$, so $\rho\circ h$ is bijective and we can take $\iota_\rho=h$.
\end{proof}

\begin{corollary}[Making a listable presentation bijective]\label{CoroBij} Let $\Mfrak$ be an infinite $\Lcal$-structure with a listable presentation $\rho$. The following are equivalent:
\begin{itemize}
\item[(i)] $\ne$ is $\rho$-listable over $\Mfrak$.
\item[(ii)]  $E_\rho$ is decidable.
\item[(iii)] There is a bijective listable presentation $\tilde{\rho}$ for $\Mfrak$ which satisfies $\tilde{\rho}\approx \rho$. 
\end{itemize}
\end{corollary}
\begin{proof} Items (i) and (ii) are equivalent by Lemma \ref{LemmaNE}.

Assuming (ii), Lemma \ref{LemmaBij} allows us to  take $\tilde{\rho}=\rho\circ \iota_\rho$. This proves (iii).

Conversely, if (iii) holds, then $E_{\tilde{\rho}}$ is decidable by Lemma \ref{LemmaBijDec}. Thus, the diagonal $\Delta\subseteq M^2$ is $\tilde{\rho}$-decidable, which implies that $\Delta$ is $\rho$-decidable by Lemma \ref{LemmaComparisonListSets}. Hence, $E_\rho=\rho^*(\Delta)$ is decidable.
\end{proof}
With Corollary \ref{CoroBij} we can give a refinement of Lemma \ref{LemmaComparisonListSets}.
\begin{theorem}[Equivalence and comparison of listable sets]\label{ThmEquivCompare} Let $\Mfrak$ be a listable $\Lcal$-structure and let $\rho,\gamma$ be listable presentations. In the following (i) implies (ii), and (ii) implies (iii):
\begin{itemize}
\item[(i)] $\rho\approx \gamma$
\item[(ii)] The class of $\rho$-listable sets is the same as the class of $\gamma$-listable sets.
\item[(iii)] The class of $\rho$-listable sets is contained in  the class of $\gamma$-listable sets.
\end{itemize}
Furthermore, if $E_\rho$ is decidable, then the three properties are equivalent.
\end{theorem}
\begin{proof} (i) implies (ii) by Lemma \ref{LemmaComparisonListSets}, and it is clear that (ii) implies (iii). It only remains to show that  (iii) implies (i) assuming that $E_\rho$ is decidable. By Proposition \ref{PropFinUL} it suffices to consider the case when $M$ is infinite.

Assume (iii) and that $E_\rho$ is decidable. By Corollary \ref{CoroBij} we may  assume that $\rho:\N\to M$ is bijective after replacing it by an equivalent listable presentation ---the class of $\rho$-listable sets in (iii) remains the same by Lemma \ref{LemmaComparisonListSets}. 

Let $T_\rho=\{(\rho(n), \rho(n+1)) : n\in \N\}\subseteq M^2$. Since $T_\rho$ is $\rho$-listable, we get that $T_\rho$ is $\gamma$-listable by (iii). Let $g=(g_1,g_2):\N\to \N^2$ be a total recursive function with image $\gamma^*(T_\rho)$. Define a function $\psi:\N\to \N$ by choosing any $\psi(0)\in \gamma^{-1}(\rho(0))$ and for $x>0$ we define
$$
\psi(x)=g_2\left(\mu y [ g_1(y)=\psi(x-1) ] \right).
$$
Then $\psi:\N\to\N$ is total recursive. Note that for each $n\ge 0$ we have $(\gamma(\psi(n)),\gamma(\psi(n+1)))\in T_\rho$. Since $\gamma(\psi(0))=\rho(0)$ and $\rho$ is bijective, we deduce  $\gamma\circ \psi=\rho$. This proves $\rho\approx \gamma$.
\end{proof}
\begin{corollary} Let $\Mfrak$ be a listable $\Lcal$-structure for which the binary relation $\ne$ is totally listable. Given $\rho$ and $\gamma$ listable presentations for $\Mfrak$, the following are equivalent:
\begin{itemize} 
\item[(i)] $\rho\approx \gamma$
\item[(ii)] The class of $\rho$-listable sets is the same as the class of $\gamma$-listable sets.
\item[(iii)] The class of $\rho$-listable sets is contained in the class of $\gamma$-listable sets.
\end{itemize}
\end{corollary}
\begin{proof} By Lemma \ref{LemmaNE} and Theorem \ref{ThmEquivCompare}.
\end{proof}

Given $\rho,\gamma:\N\to M$ listable presentations of an $\Lcal$-structure $\Mfrak$, we define
$$
\Delta(\gamma,\rho):=\{(m,n)\in \N^2 : \gamma(m)=\rho(n)\}\subseteq \N^2.
$$
In particular, note that $E_\rho=\Delta(\rho,\rho)$.
\begin{lemma}[Diagonal test for equivalence]\label{LemmaDTE} Let $\gamma, \rho:\N\to M$ be listable presentations for $\Mfrak$. We have that $\gamma\approx\rho$ if and only if $\Delta(\gamma,\rho)\subseteq \N^2$ is listable.
\end{lemma}
\begin{proof} Assume $\rho\approx \gamma$ and let $\psi:\N\to \N$ be a total recursive function with $\gamma\circ \psi=\rho$. Let $\epsilon^\gamma=(\epsilon^\gamma_1,\epsilon^\gamma_2):\N\to \N^2$ be a total recursive map with image $E_\gamma$. Define the partial function
$$
\delta(m,n)=\mu y[\epsilon^\gamma_2(y)=m \mbox{ and } \epsilon^\gamma_1(y)=\psi(n)  ].
$$
The function $\delta$ is partial recursive and its domain is $\Delta(\gamma,\rho)$. Thus, $\Delta(\gamma,\rho)$ is listable.

Conversely, assume that $\Delta(\gamma,\rho)$ is listable. It is non-empty, so, there is a total recursive $f=(f_1,f_2):\N\to \N^2$ with image $\Delta(\gamma,\rho)$. Observe that both $f_1,f_2:\N\to \N$ are total recursive and surjective. Define
$$
\phi(n)=f_2(\mu y [f_1(y)=n])
$$
Then $\phi:\N\to \N$ is total recursive and it satisfies $\rho(\phi(n))=\gamma(n)$ for all $n\ge 0$. Hence $\rho\approx \gamma$.
\end{proof}
In view of Lemma \ref{LemmaListtoInt} we get
\begin{corollary} If $\rho,\gamma:\N\to M$ are listable presentations for an $\Lcal$-structure $\Mfrak$, we have $\rho\approx\gamma$ if and only if $\rho\asymp\gamma$ as p.e. interpretations. 
\end{corollary}
\begin{proof} Note that in this case $K(\gamma, \rho)=\Delta(\gamma,\rho)$. By the DPRM theorem, we see that $K(\gamma, \rho)$ is p.e. $\Lcal_a$-definable over $\N$ if and only if $\Delta(\gamma,\rho)$ is listable. We conclude by Lemma \ref{LemmaDTE}.
\end{proof}

\subsection{Uniquely listable structures} We have seen that listable structures with finite domain are uniquely listable (cf. Proposition \ref{PropFinUL}). However, not all listable structures are uniquely listable. For instance, let $H\subseteq \N$ be a listable undecidable set and consider the structure $(\N;H,=)$. The identity map $\rho:\N\to \N$ is a listable presentation. Another listable presentation is given by the set-theoretical bijection $\gamma:\N\to \N$ mapping $2n$ to the $n$-th element of $H$ and $2n+1$ to the $n$-th element of $H^c$. Since $H^c$ is $\gamma$-listable but it is not $\rho$-listable, we conclude $\rho\not\approx \gamma$.

Let us discuss the problem of determining whether a listable structure is uniquely listable. First, we have the following basic transference property.

\begin{lemma}[Transference of unique listability] \label{LemmaUnique} Let $\Lcal$ and $\Kcal$ be languages. Let $\Mfrak$ be a listable $\Lcal$-structure with domain $M$, and let $\Nfrak$ be a uniquely listable $\Kcal$-structure with domain $N$. Suppose that there is a bijective function $\theta: M\to N$ defining a p.e. interpretation $\theta:\Mfrak\dasharrow \Nfrak$. Then $\Mfrak$ is uniquely listable.
\end{lemma}
\begin{proof} Given $\rho,\gamma:\N\to M$ listable presentations for $\Mfrak$ we have that $\theta\circ\rho$ and $\theta\circ\gamma$ are listable presentations of $\Nfrak$ (by Proposition \ref{PropIntimplList} with $f=\Id_\N$) and therefore they are equivalent. Let $\phi:\N\to\N$ be a total recursive function with $\theta\circ \rho\circ \phi=\theta\circ\gamma$, then $\rho\circ \phi=\gamma$ because $\theta$ is injective, and we get $\rho\approx \gamma$. 
\end{proof}

Unfortunately, this transference property is rather restrictive and a more flexible criterion for unique listability is necessary.

Let $\Mfrak$ be a listable $\Lcal$-structure with domain $M$. A \emph{universal listing} for $\Mfrak$ is a surjective set-theoretical function $\tau:\N\to M$ satisfying that for every listable presentation $\rho:\N\to M$ there is a total recursive function $a_{\rho}^{\tau}:\N\to \N$ such that $\rho\circ a_{\rho}^\tau = \tau$.

Universal listings are relevant for us due to the following relation with unique listability.
\begin{lemma}[Universal listings and unique listability]\label{LemmaUL} Let $\Mfrak$ be a listable $\Lcal$-structure which admits a universal listing $\tau$. Then $\Mfrak$ is uniquely listable and $\tau$ is a listable presentation for it.
\end{lemma}
\begin{proof}  Let $\rho$ be any listable presentation for $\Mfrak$. By assumption, $\tau$ is surjective. Since $\tau=\rho\circ a^\tau_\rho$ and $a^\tau_\rho:\N\to \N$ is total recursive, we see that for every $s\in \Lcal$ the set $\tau^*(s^\Mfrak)=(a^\tau_\rho)^*(\rho^*(s^{\Mfrak}))$ is listable. Hence, $\tau$ is a listable presentation for $\Mfrak$ and the relation $\tau=\rho\circ a^\tau_\rho$ implies $\rho\approx \tau$.
\end{proof}

Naturally, there is the problem of showing that a given structure actually has some universal listing. For this we have:

\begin{theorem}[Criterion for unique listability]\label{ThmUL} Let $\Mfrak$ be a listable $\Lcal$-structure. Let $r,k$ be positive integers and let $c\in \N$. Let us choose the following:
\begin{itemize}
\item  a total recursive function $h=(h_1,...,h_r):\N\to \N^r$ such  that $h_j(n)<n$ for each $j=1,...,r$ and all $n>c$, 
\item a partition $A_1,...,A_k$ of $\N_{>c}$ with each $A_i$ decidable,
\item for each $i=1,2,...,k$, a p.e. $\Lcal$-definable function $F_i: V_i\to M$ with domain $V_i\subseteq M^r$. 
\end{itemize}
Let $\tau:\N\to M$ be a set-theoretical function satisfying the following conditions:
\begin{itemize}
\item[(i)] $\tau$ is surjective.
\item[(ii)] If $n\in A_i$ and $n>c$, then $h(n)\in \tau^* (V_i)$.
\item[(iii)] For each $n>c$, we have
$$
\tau(n)=\begin{cases}
F_1(\tau(h_1(n)),...,\tau(h_r(n)))\mbox{ if }n\in A_1\\
\vdots\\
F_k(\tau(h_1(n)),...,\tau(h_r(n)))\mbox{ if }n\in A_k. 
\end{cases}
$$
\end{itemize}
Then $\tau$ is a universal listing. In particular, $\Mfrak$ is uniquely listable and $\tau$ is a listable presentation.
\end{theorem}
\begin{proof} Let $\rho:\N\to M$ be a listable presentation. Let 
$$
\Gamma_i=\{(x_0,...,x_r):  x_0=F_i(x_1,...,x_r)\}\subseteq M^{r+1}
$$ 
and note that $\Gamma_j$ is p.e. $\Lcal$-definable. By Corollary \ref{CoroTotList} the set $\rho^*(\Gamma_i)\subseteq \N^{r+1}$ is listable. For each $i=1,...,k$ let $f_i=(f_{i0},...,f_{ir}):\N\to \N^{r+1}$ be a total recursive map with image $\rho^*(\Gamma_i)$. We note that $(f_{i1},...,f_{ir}):\N\to \N^r$ has image $\rho^*(V_i)$ because $V_i$ is the domain of $F_i$.

We claim that there is a total recursive function $\alpha:\N\to \N$ satisfying the following:
\begin{itemize}
\item[(a)]  For each $n\in \N$ we have $\rho (\alpha(n)) = \tau(n)$.
\item[(b)] For each $n\in A_i$ with $n>c$ we have $(\alpha(h_1(n)),...,\alpha(h_r(n)))\in \rho^*(V_i)$.
\item[(c)] For each $n>c$,
\begin{equation}\label{EqnDefalpha}
\alpha(n)=\begin{cases}
f_{10}(\mu y [ f_{1j}(y)=\alpha(h_j(n)) \mbox{ for each }j=1,...,r] )\mbox{ if }n\in A_1\\
\vdots\\
f_{k0}(\mu y [ f_{kj}(y)=\alpha(h_j(n)) \mbox{ for each }j=1,...,r] )\mbox{ if }n\in A_k.
\end{cases}
\end{equation}
\end{itemize}
Let us choose any values $\alpha(n)\in \rho^{-1}(\tau(n))$ for $n\le c$. Then $\rho(\alpha(m))=\tau(m)$ for $m=0,1,...,c$. We will recursively construct the values of the function $\alpha(n)$ for larger values of $n$ by using (c), and along the construction we will inductively prove that (a) and (b) hold.

Let us fix an $n>c$ and let us assume that for each $m<n$ we have that $\alpha(m)\in \N$ is already defined and that $\rho(\alpha(m))=\tau(m)$ holds. Let $i$ be the index with $n\in A_i$. Note that $h_j(n)<n$ for each $j=1,...,r$, hence $(\alpha(h_1(n)),...,\alpha(h_r(n)))\in \N^r$ is already defined and 
\begin{equation}\label{Eqnu1}
(\rho(\alpha(h_1(n))),...,\rho(\alpha(h_r(n)))) = (\tau(h_1(n)),...,\tau(h_r(n)))\in V_i
\end{equation}
by condition (ii). That is,
$$
(\alpha(h_1(n)), ..., \alpha(h_r(n)))\in \rho^*(V_i)\subseteq \N^r
$$
as required by (b).

Since $\rho^*(V_i)$ is the image of $(f_{i1},...,f_{ir}):\N\to \N^r$, there is some $y\in \N$ such that $f_{ij}(y)=\alpha(h_j(n))$ for each $j=1,...,r$. Therefore, \eqref{EqnDefalpha} uniquely defines $\alpha(n)\in \N$.  Furthermore, if $y_0$ is the minimal such $y$ for our chosen $n$, then 
$$
\begin{aligned}
\rho(\alpha(n))& =\rho(f_{i0}(y_0)) = F_i(\rho(f_{i1}(y_0)),...,\rho(f_{ir}(y_0)))\quad &\mbox{ by definition of }f_i:\N\to \N^{r+1}\\
& = F_i(\rho(\alpha(h_1(n))),...,\rho(\alpha(h_r(n)))) \quad &\mbox{ by choice of }y_0\\
&= F_i(\tau(h_1(n)),...,\tau(h_r(n))) \quad &\mbox{ by \eqref{Eqnu1}}\\
&= \tau(n) \quad &\mbox{ by condition (iii).}
\end{aligned}
$$
This proves that $\rho(\alpha(n))=\tau(n)$ holds, as required by (a). 

Finally, it only remains to observe that the function  $\alpha:\N\to \N$ that we have constructed is total recursive. In fact, we already proved that $\alpha$ is a total function, and it is defined by the first chosen values $\alpha(1),...,\alpha(c)$ together with the condition \eqref{EqnDefalpha} for $n>c$. The condition \eqref{EqnDefalpha} shows that $\alpha$ is recursive because it only involves the total recursive functions $f_{ij}$ and $h_j$, the schema of definition by decidable cases, the minimalization operator, and the schema of course-of-values recursion. 

In particular, we have constructed a total recursive function $\alpha:\N\to \N$ satisfying condition (a). The function $\tau:\N\to M$ is surjective by assumption (i), and the choice $a^\tau_\rho=\alpha$ shows that $\tau$ is a universal listing. We conclude by Lemma \ref{LemmaUL}.
\end{proof}

\subsection{Examples}\label{SecULexamples}

A first example to explain how to use the previous results on unique listability:

\begin{proposition}\label{PropExN} Let $\Nfrak$ be an $\Lcal$-structure with domain $\N$ such that
\begin{itemize}
\item[(i)] For each $s\in \Lcal$, the set $s^{\Nfrak}$ is listable.
\item[(ii)] $0\in \N$ and the successor function $S:\N\to \N$, $S(x)=x+1$ are p.e. $\Lcal$-definable in $\Nfrak$.
\end{itemize}
Then $\Nfrak$ is uniquely listable. In particular, this holds for  $(\N;0,S,=)$ and $(\N;0,1,+,\times,=)$.
\end{proposition}
\begin{proof} The identity function gives a listable presentation by (i). It remains to prove uniqueness up to equivalence. We apply Theorem \ref{ThmUL} with $c=0$, $r=k=1$, $A_1=\N$, $h(n)=\max\{0,n-1\}$, $F_1=S$, $V_1=\N$, and $\tau=\Id_\N$. The result follows.
\end{proof}
Next we consider the case of $\Q$ in detail. We begin with a folklore fact.
\begin{lemma}\label{LemmaQnum} Let $q: \Z_{>0}\to \Q_{>0}$ be the function defined by  $q(1)=1$ and for $n\ge 2$:
$$
q(n)=\begin{cases}
q(n/2) + 1&\mbox{ if $n\ge 2$ is even}\\
1/q(n-1) &\mbox{ if $n\ge 2$ is odd.}
\end{cases}
$$
Then $q: \Z_{>0}\to \Q_{>0}$ is bijective. 
\end{lemma}
\begin{proof}  Take any rational number $r>0$ and recall that it admits a unique continued fraction expansion $[a_0;a_1,...,a_d]$ for some $d\ge 1$ under the requirements $a_0\ge 0$, $a_j\ge 1$ for $j\ge 1$, and $a_d=1$. We note that choosing the positive integer 
\begin{equation}\label{Eqnbinary}
n=2^{a_0}(2^{a_1}(...(2^{a_{d-2}}(2^{a_{d-1}}+1)+1)...)+1)
\end{equation} 
we get $q(n)=r$, so $q(\Z_{>0})=\Q_{> 0}$. Furthermore, $q$ is injective because the expression \eqref{Eqnbinary} always exists and is unique for a given $n\ge 1$ under the requirements $a_0\ge 0$ and $a_j\ge 1$ for $j\ge 1$ ---expanding the product in \eqref{Eqnbinary} we get the binary expansion of $n$.
\end{proof}
\begin{corollary}\label{CoroQtau} Let $\tau:\N\to \Q$ be defined by $\tau(0)=0$, $\tau(1)=1$, $\tau(2)=-1$ and for $n\ge 3$
$$
\tau(n)=\begin{cases}
\tau((n-1)/2)+1 	&\mbox{ if }n\equiv 3\bmod 4\\
\tau(n/2)-1  	&\mbox{ if }n\equiv 0\bmod 4\\
1/\tau(n-2)		&\mbox{ if }n\equiv 1,2\bmod 4.
\end{cases}
$$
Then $\tau:\N\to \Q$ is bijective.
\end{corollary}
\begin{proof} This follows from Lemma \ref{LemmaQnum}, since the sequence of values $\tau(n)$ for $n=0,1,2,...$ is precisely $0$, $q(1)$, $-q(1)$, $q(2)$, $-q(2)$, ...
\end{proof}
\begin{proposition} \label{PropULQ} Let $\Qfrak$ be an $\Lcal$-structure with domain $\Q$ such that
\begin{itemize}
\item[(i)] $\Qfrak$ admits some listable presentation.
\item[(ii)] The constant $0\in \Q$, the successor function $S:\Q\to \Q$, $S(x)=x+1$, and the multiplicative inverse function $R:\Q^\times \to \Q$, $R(x)=1/x$ are p.e. $\Lcal$-definable in $\Qfrak$.
\end{itemize}
Then $\Qfrak$ is uniquely listable and the function $\tau:\N\to \Q$ of Corollary \ref{CoroQtau} is a listable presentation.
\end{proposition}
\begin{proof}  First we note that the function $P:\Q\to\Q$ given by $y=P(x)=x-1$ is p.e. $\Lcal$-defined by the formula $x=S(y)$.

 Let us  apply Theorem \ref{ThmUL} to show that the function $\tau:\N\to \Q$ from Corollary \ref{CoroQtau} is a universal listing for $\Qfrak$. We choose $c=2$, $r=1$, $k=3$, $A_1=\{n\ge 3 : n\equiv 3\bmod 4\}$, $A_2=\{n\ge 3 : n\equiv 0\bmod 4\}$, $A_3=\{n\ge 3 : n\equiv 1,2\bmod 4\}$, and the total recursive function $h:\N\to \N$ given by
$$
h(n)=\begin{cases}
(n-1)/2 &\mbox{ if }n\equiv 3\bmod 4\\
n/2 &\mbox{ if } n\equiv 0\bmod 4\\
\max\{0, n-2\} &\mbox{ if }n\equiv 1,2\bmod 4.
\end{cases}
$$
With these choices, the function $\tau:\N\to \Q$ from Corollary \ref{CoroQtau} satisfies $\tau(0)= 0$, $\tau(1)= S(0)$, $\tau(2)=P(0)$ and for $n\ge 3$ we have
$$
\tau(n)=\begin{cases}
S(h(n))\mbox{ if } n\in A_1\\
P(h(n))\mbox{ if } n\in A_2\\
R(h(n))\mbox{ if } n\in A_3.
\end{cases}
$$
Theorem \ref{ThmUL} implies that $\tau$ is a universal listing. Hence the result.
\end{proof}

\begin{corollary}\label{CoroQulist} The $\Lcal_a$-structure $\Q$ is uniquely listable.
\end{corollary}
\begin{proof} By Corollary \ref{CoroListableZQ} and Proposition \ref{PropULQ}.
\end{proof}

In a similar fashion, one can get several other results. We include here the case of a certain class of finitely generated structures, which generalizes the examples we have presented in detail so far in this section. The proofs goes along the same lines.
\begin{proposition}\label{PropFG} Let $\Mfrak$ be a listable $\Lcal$-structure with domain $M$. Suppose that there is a finite list of elements $g_0, g_1, ...,g_c\in M$ and functions $F_1,...,F_k:M^r\to M$ such that
\begin{itemize}
\item  each $g_j$ and each $F_i$ is p.e. $\Lcal$-definable over $\Mfrak$, and
\item $M$ is generated by the elements $g_j$ and the functions $F_i$, in the sense that each element of $M$ is obtained by applying a suitable composition of the $F_i$ to the elements $g_j$. 
\end{itemize}
Then $\Mfrak$ is uniquely listable.
\end{proposition}

One immediately gets the following generalization of Corollary \ref{CoroQulist}:
\begin{corollary}\label{CoroULfg}  Let $M$ be a finitely generated ring. Let $\Lcal=\Lcal_a\cup\{g_0,g_1,...,g_c\}$ where $g_i$ are symbols of constant interpreted in $M$ as a list of ring  generators. Then $M$ seen as an $\Lcal$-structure is uniquely listable. The same result holds if $M$ is a finitely generated field, in which case the $g_i$ should be taken to correspond to a list of field generators.
\end{corollary}
\begin{proof} It is a standard fact that finitely generated rings and fields are p.e. interpretable in the $\Lcal_a$-structure $\N$; see for instance Corollary 2.14 in \cite{AKNS} and note that the constructed interpretation is p.e. (idea: it suffices to interpret rings of the form $\Z[x_1,...,x_n]$ and then take quotients and localizations, which leads to p.e. interpretations). The result follows from Theorem \ref{ThmListableN} and Proposition \ref{PropFG}.
\end{proof}



\section{The DPRM property for listable structures}\label{SecDPRM}

\subsection{The DPRM property}\label{SecDPRM1}  Let $\Mfrak$ be a listable $\Lcal$-structure. Recall that a set $X\subseteq M^r$ is totally listable if for every listable presentation $\rho:\N\to M$ we have that $X$ is $\rho$-listable. 

In the special case that $\Mfrak$ is uniquely listable, the class of totally listable sets over $\Mfrak$ is the same as the class of $\rho$-listable sets for any chosen listable presentation $\rho:\N\to M$ (cf. Lemma \ref{LemmaComparisonListSets}). So, if $\Mfrak$ is uniquely listable, there is a well-defined notion of listable sets over $\Mfrak$.

By Corollary \ref{CoroTotList}, if $X\subseteq M^r$ is p.e. $\Lcal$-definable, then it is totally listable. We are interested in listable structures $\Mfrak$ for which the converse holds.

We say that a listable $\Lcal$-structure $\Mfrak$ has the \emph{DPRM property} if for each $r\ge 1$, every totally listable set $X\subseteq M^r$  is p.e. $\Lcal$-definable. Thus, a listable structure $\Mfrak$ has the DPRM property if and only if the class of totally listable sets over $\Mfrak$ is the same as p.e. $\Lcal$-definable sets over $\Mfrak$.

Of course, $\N$ as an $\Lcal_a$-structure has the DPRM property precisely by the DPRM theorem. It easily follows that $\Z$ as an $\Lcal_a$-structure also has the DPRM property (see Lemma \ref{LemmaZDPRM} for details).

A basic feature of listable structures having the DPRM property is the following:

\begin{lemma} If $\Mfrak$ is a listable $\Lcal$-structure with the DPRM property, then all finite sets in $M^r$ are p.e. $\Lcal$-definable. In particular, each element of $M$ is p.e. $\Lcal$-definable.
\end{lemma}
\begin{proof} By Corollary \ref{CoroFinSets}.
\end{proof}
In particular, the DPRM property has a simple characterization for finite structures.
\begin{corollary} Let $\Mfrak$ be an $\Lcal$-structure with finite domain. Then $\Mfrak$ has the DPRM property if and only if each element of $M$ is p.e. $\Lcal$-definable. In this case, every subset of $M^r$ is p.e. $\Lcal$-definable.
\end{corollary}

\subsection{Known results}\label{SecDPRMknown}

Other authors have considered a slightly different variant of the DPRM property where only \emph{recursive presentations} are allowed. That is, surjective maps $\rho:\N\to M$ such that for every $s\in \Lcal$ we have that $\rho^*(s^\Mfrak)$ is decidable. Let us momentarily call this variant the \emph{recursive DPRM property}. Note that if a set $X\subseteq M^r$ is $\rho$-listable for every listable presentation $\rho$, then the same holds for every recursive presentation. Therefore, the recursive DPRM property implies the DPRM property, and all available results in the literature establishing the recursive DPRM property for a certain structure are still valid if the DPRM property is considered.

Let us briefly survey the known cases of the DPRM property beyond $\N$ and $\Z$. Denef \cite{DenefZT} proved the DPRM property for $\Z[t]$. Zahidi \cite{ZahidiOKt} extended Denef's result to the polynomial  ring $O_K[t]$ when $K$ is a totally real number field and $O_K$ is its ring of integers. Demeyer proved several other cases: $k[t]$ for $k$ a finite field or a recursive infinite algebraic extension of a finite field \cite{DemeyerInv}, and $A[t]$ where $A$ is a recursive subring of a number field \cite{DemeyerPoly}.  Degroote and Demeyer \cite{DeDe} proved the case of $L[t]$ where $L$ is an automorphism-recursive algebraic extension of $\Q$ having an embedding into $\R$ or into a finite extension of $\Q_p$ for some prime $p$ (see \cite{DeDe} for the definition of automorphism-recursive extensions). 

Other than this, if $B$ is a uniquely listable ring and $A\subseteq B$ is a uniquely listable subring such that $A$ has the DPRM property and $A$ is Diophantine in $B$ (that is, p.e. $\Lcal_a$-definable with parameters), then there are a number of cases where the DPRM property can be transferred from $A$ to $B$ by applying some version of Theorem \ref{ThmTransferDPRM} below, especially in the context of finite algebraic extensions of fields or Dedekind domains. See for instance Proposition \ref{PropOK} below.

\subsection{The number of existential quantifiers}\label{SecEquant} Let $\Mfrak$ be an $\Lcal$-structure.  Given $r\ge 1$, a \emph{p.e. $r$-catalogue} for $\Mfrak$ is a p.e. $\Lcal$-formula $\Upsilon_r[x_0,x_1,...,x_r]$   with the following property: For every p.e. $\Lcal$-definable set $X\subseteq M^r$ there is an element $m_X\in M$ such that $X=\{{\bf a}\in M^r : \Mfrak\models \Upsilon_r[m_X,{\bf a}]\}$.
(One may argue that ``universal p.e. formula'' would be a better terminology, but such an oxymoron might lead to confusion.) We record the following remark.
\begin{lemma} Let $\Mfrak$ be an $\Lcal$-structure. If $\Mfrak$ has a p.e. $r$-catalogue for certain $r\ge 1$, then it has  a p.e. $n$-catalogue for every $1\le n\le r$.
\end{lemma}

Conversely, one has:
\begin{lemma} Let $\Mfrak$ be an $\Lcal$-structure with domain $M$. If $\Mfrak$ has a p.e. $1$-catalogue and there is a p.e. $\Lcal$-definable injective function $M^2\to M$, then $\Mfrak$ has a p.e. $r$-catalogue for every $r\ge 1$.
\end{lemma}

For an $\Lcal$-structure $\Mfrak$ and subset $X\subseteq M^r$ which is p.e. $\Lcal$-definable with parameters from $\Mfrak$,  we define the \emph{p.e. rank} of $X$ over $\Mfrak$ as the minimal number of  existential quantifiers required by a p.e. $\Lcal$-formula to define $X$ \emph{with parameters from $\Mfrak$}. The p.e. rank of $X$ over $\Mfrak$ is denoted by  $\rk^{p.e.}_{\Mfrak}(X)$.  We allow p.e. definitions with parameters because the p.e. rank is intended to be a rough measure of the complexity of a p.e. definable set, and it is desirable that all p.e. definable finite sets have the lowest possible complexity in this sense ---namely, $0$.

The following observation follows from the definition of p.e. $r$-catalogues.
\begin{lemma}[Boundedness of the p.e. rank]\label{LemmaCatBdd} Let $\Mfrak$ be an $\Lcal$-structure with domain $M$ and let $r\ge 1$.  If $\Mfrak$ has a p.e. $r$-catalogue, then there is a bound $B_\Mfrak(r)$ depending only on $\Mfrak$ and $r$, such that for every $n\le r$ and every p.e. $\Lcal$-definable $X\subseteq M^n$ we have $\rk^{p.e.}_{\Mfrak}(X)\le B_{\Mfrak}(r)$.
\end{lemma}
\begin{theorem}[Existence of p.e. catalogues] \label{ThmCatDPRM} Let $\Mfrak$ be an $\Lcal$-structure with infinite domain $M$. Suppose that $\Mfrak$ is uniquely listable, that it has the DPRM property, and that the binary relation $\ne$ is p.e. $\Lcal$-definable over $\Mfrak$. Then $\Mfrak$ has a p.e. $r$-catalogue $\Upsilon_r[x_0,x_1,...,x_r]$ for every $r\ge 1$.

In particular, for every $r\ge 1$ there is a bound $B_\Mfrak(r)$ (depending only on $\Mfrak$ and $r$) such that for every p.e. $\Lcal$-definable set $X\subseteq M^r$ we have $\rk^{p.e.}_{\Mfrak}(X)\le B_\Mfrak(r)$.
\end{theorem}
\begin{proof} Take any listable presentation $\rho:\N\to M$. Since $M$ is infinite and $\ne$ is p.e. $\Lcal$-definable, Corollary \ref{CoroBij} yields a bijective listable presentation $\gamma:\N\to M$. Lemma \ref{LemmaUnivList} then gives a $\gamma$-listable set $U_r(\gamma)\subseteq M^{r+1}$ such that for every $\gamma$-listable set $X\subseteq M^r$ there is $m_X\in M$ such that
$$
X = \{{\bf a}\in M^r : (m_X,{\bf a})\in U_r(\gamma)\}.
$$
Since $\Mfrak$ has the DPRM property and it is uniquely listable, $U_r(\gamma)$ is p.e. $\Lcal$-definable. We can take $\Upsilon_r[x_0,x_1,...,x_r]$ to be any p.e. $\Lcal$-formula defining $U_r(\gamma)$.
\end{proof}

\subsection{Transference of the DPRM property}  Our next goal is to characterize when it is possible to transfer the DPRM property from one structure to another by means of a p.e. interpretation. For simplicity, we will work with uniquely listable structures. First we need

\begin{lemma}\label{LemmaSelf} Let $\Mfrak$ be a uniquely listable $\Lcal$-structure with domain $M$ and suppose that it has the DPRM property. Then for every p.e. self-interpretation $\lambda:\Mfrak\dasharrow \Mfrak$ we have that 
$$
\Gamma(\lambda) =\{(x_0,x_1,...,x_r)\in M^{r+1} : (x_1,...,x_r)\in \dom(\lambda)\mbox{ and } x_0=\lambda(x_1,...,x_r)\}  
$$
is p.e. $\Lcal$-definable.
\end{lemma}
\begin{proof} Since $\dom(\lambda)$ is p.e. $\Lcal$-definable, it is totally listable (cf. Corollary \ref{CoroTotList}). Let $\rho$ be a listable presentation for $\Mfrak$ and let $f=(f_1,...,f_r):\N\to \N^r$ be a total recursive function with image $\rho^*(\dom(\lambda))$. The map $\gamma=\lambda\circ \rho^{(r)}\circ f:\N\to M$ is a listable presentation for $\Mfrak$, by Proposition \ref{PropIntimplList}. Since $\Mfrak$ is uniquely listable, $\gamma\approx \lambda$ and there is a total recursive function $\phi:\N\to \N$ with $\gamma=\rho\circ \phi$. Note that $\rho^{(r)}\circ f:\N\to M^r$ maps onto $\dom(\lambda)$, so, $\Gamma(\lambda)$ consists precisely of elements of the form 
$$
(\gamma(n), \rho^{(r)}(f(n))) = (\rho(\phi(n)) , \rho^{(r)}(f(n))) = \rho^{(r+1)}\circ(\phi,f_1,...,f_r)(n)
$$
for some $n\in \N$. We conclude by Lemma \ref{LemmaCharList} (iii) and the DPRM property.
\end{proof}

\begin{theorem}[Transference of DPRM]\label{ThmTransferDPRM} For $i=1,2$ let $\Lcal_i$  be a language and let $\Mfrak_i$ be a uniquely listable $\Lcal_i$-structure with domain $M_i$. Suppose that $\Mfrak_2$ has the DPRM property and  that there are p.e. interpretations $\theta_1:\Mfrak_1\dasharrow \Mfrak_2$ and $\theta_2:\Mfrak_2\dasharrow \Mfrak_1$ of ranks $r_1,r_2$ respectively. Let $\zeta=\theta_2\bullet \theta_1:\Mfrak_1\dasharrow\Mfrak_1$ and let $r=r_1r_2$. Then the following are equivalent:
\begin{itemize}
\item[(i)] The pair $(\theta_1,\theta_2)$ defines a p.e. bi-interpretation between $\Mfrak_1$ and $\Mfrak_2$.
\item[(ii)] $\Gamma(\zeta)\subseteq M_1^{r+1}$ is p.e. $\Lcal_1$-definable over $\Mfrak_1$.
\item[(iii)] $\Mfrak_1$ has the DPRM property.
\end{itemize}
\end{theorem}
\begin{proof} (i) implies (ii) by definition of p.e. bi-interpretation.

Let us show that (ii) implies (iii). Let $\rho$ be a listable presentation for $\Mfrak_2$. By Proposition \ref{PropIntimplList} applied to $\theta_2:\Mfrak_2\dasharrow\Mfrak_1$, there is a listable presentation $\gamma$ for $\Mfrak_1$ such that for any $S\subseteq M_1^k$ we have that $S$ is $\gamma$-listable if and only if $\theta_2^*(S)$ is $\rho$-listable.

Let $S\subseteq M_1^k$ be a $\gamma$-listable set. Then $\theta_2^*(S)$ is $\rho$-listable. Since $\Mfrak_2$ is uniquely listable and it has the DPRM property, $\theta_2^*(S)$ is p.e. $\Lcal_2$-definable over $\Mfrak_2$. Hence, the set $\zeta^*(S)=\theta_1^*(\theta_2^*(S))\subseteq M_1^{rk}$ is p.e. $\Lcal_1$-definable over $\Mfrak_1$.  Since $\zeta:\dom(\zeta)\to M_1$ is surjective, we have $S=\zeta^{(k)}(\zeta^*(S))\subseteq M_1^k$. By (ii), the function $\zeta^{(k)}:\dom(\zeta)^k\to M_1^k$ is p.e. $\Lcal_1$-definable. Therefore, the set
$$
S=\{{\bf y}\in M_1^k : \exists {\bf x}\in \zeta^*(S)\mbox{ such that } \zeta^{(k)}({\bf x})={\bf y}\}
$$
is p.e. $\Lcal_1$-definable over $\Mfrak_1$. This proves (iii). Finally, (iii) implies (i) by Lemma \ref{LemmaSelf}.
\end{proof}
\begin{corollary}\label{CoroDPRM} Let $\Mfrak$ be a uniquely listable $\Lcal$-structure with domain $M$ and let $\rho$ be a listable presentation for it. Suppose that we have an interpretation $\theta:\Mfrak\dasharrow \N$ with $\N$ is regarded as an $\Lcal_a$-structure. The following are equivalent:
\begin{itemize}
\item[(i)] $\Mfrak$ is p.e. bi-interpretable with the $\Lcal_a$-structure $\N$.
\item[(ii)] The function $\rho\circ \theta:\dom(\theta)\to M$ is p.e. $\Lcal$-definable.
\item[(iii)]   $\Mfrak$ has the DPRM property.
\end{itemize}
\end{corollary}
\begin{proof} Recall that the $\Lcal_a$-structure $\N$ is uniquely listable (cf. Proposition \ref{PropExN}) and it has the DPRM property. Theorem \ref{ThmTransferDPRM} shows that (i) for any bi-interpretation implies (iii).

 By Lemma \ref{LemmaListtoInt}, $\rho$ defines a p.e. interpretation of $\Mfrak$ in the $\Lcal_a$-structure $\N$.  Applying Theorem \ref{ThmTransferDPRM}  with $\Lcal_1=\Lcal$, $\Lcal_2=\Lcal_a$, $\Mfrak_1=\Mfrak$, $\Mfrak_2=\N$,  $\theta_1=\theta$, and $\theta_2=\rho$, we see  that  (ii) is equivalent to (iii) and they imply (i).
\end{proof}
The equivalence between (ii) and (iii) in Corollary \ref{CoroDPRM} has been previously used in the literature to transfer the DPRM property from the semi-ring $\N$ to recursive rings and fields; see for instance \cite{DenefZT,ZahidiOKt} and, more generally, the references in Section \ref{SecDPRMknown}. We also refer the reader to Demeyer's thesis \cite{DemeyerThesis} where the equivalence between (ii) and (iii) in Corollary \ref{CoroDPRM} is discussed in the context of recursively presented rings.  

In this special form, the strategy originated in Denef's work \cite{DenefZT} where he transferred the DPRM property from $\N$ to $\Z[T]$ by implicitly using the fact that (ii) implies (iii) in Corollary \ref{CoroDPRM}. A recursive presentation for a ring satisfying (ii) of Corollary \ref{CoroDPRM} is often referred to as a \emph{Diophantine enumeration}, but they are difficult to obtain in general. We will use the more flexible criterion given by Theorem \ref{ThmTransferDPRM} in the examples of Section \ref{SecExamplesDPRM}. Nevertheless, Corollary \ref{CoroDPRM} allows us to give a characterization of uniquely listable structures with infinite domain that have the DPRM property, under the mild assumption that $\ne$ is totally listable.

\begin{theorem}[Characterization of the DPRM property for infinite uniquely listable  structures] \label{ThmCharDPRM} Let $\Mfrak$ be a uniquely listable $\Lcal$-structure with infinite domain, and with the property that $\ne$ is totally listable over $\Mfrak$ (this is the case, for instance, if $\ne$ is p.e. $\Lcal$-definable over $\Mfrak$). Then $\Mfrak$ has the DPRM property if and only if it is p.e. bi-interpretable with the $\Lcal_a$-structure $\N$.
\end{theorem}
\begin{proof} If there is a p.e. bi-interpretation, then in particular there is a p.e. interpretation $\theta:\Mfrak\dasharrow \N$ and we can apply Corollary \ref{CoroDPRM} to conclude that $\Mfrak$ has the DPRM property.

Conversely, assume that $\Mfrak$ has the DPRM property. By Corollary \ref{CoroDPRM}, it suffices to construct a p.e. interpretation $\theta:\Mfrak\dasharrow \N$ with $\N$ seen as an $\Lcal_a$-structure. Since $\Mfrak$ has infinite domain $M$ and $\ne$ is totally listable, Corollary \ref{CoroBij} implies that there is a bijective listable presentation $\rho:\N\to M$. Let $\theta=\rho^{-1}:M\to \N$. For $s\in \Lcal_a$ we have that $s^\N$ is listable, hence $\theta^*(s^\N)$ is $\rho$-listable because $\rho^*(\theta^*(s^\N))=s^\N$. By the DPRM property on $\Mfrak$ we get that $ \theta^*(s^\N)$ is p.e. $\Lcal$-definable over $\Mfrak$, which implies that $\theta:\Mfrak\dasharrow \N$ is the required p.e. interpretation.
\end{proof}

We remark that Theorem \ref{ThmCharDPRM} applies in the setting of Theorem \ref{ThmCatDPRM}.

\subsection{Examples}\label{SecExamplesDPRM} All the structures considered in this section are uniquely listable, thanks to the results in Section \ref{SecULexamples}.  If $\Mfrak$ is an $\Lcal$-structure with domain $M$ and $\Scal\subseteq M$, we let $\Lcal\cup\Scal$ be the language obtained by expanding $\Lcal$ with constant symbols corresponding to the elements of $\Scal$ and interpreted in $\Mfrak$ accordingly. As a warm-up, we have:
\begin{lemma}\label{LemmaZDPRM} The $\Lcal_a$-structure $\Z$ has the DPRM property.
\end{lemma}
\begin{proof} By Lemma \ref{LemmaIntNZ} and the equivalence of (i) and (iii) in Corollary \ref{CoroDPRM}.
\end{proof}

The next three examples are considered folklore results (except for the assertions about p.e. bi-interpretability), although the author is not aware of any reference for their proofs. Let us recall that for a ring $A$, a subset $S\subseteq A^r$ is \emph{Diophantine} if it is p.e. $\Lcal_a$-definable over $A$ with parameters. 

\begin{proposition}\label{PropZQ} The following are equivalent:
\begin{itemize}
\item[(i)] $\Z$ is Diophantine in $\Q$.
\item[(ii)]  The $\Lcal_a$-structure $\Q$ has the DPRM property.
\item[(iii)] The $\Lcal_a$-structure $\Q$ is p.e. bi-interpretable with the $\Lcal_a$-structure $\N$.
\end{itemize}
\end{proposition}
\begin{proof} Let us consider $\Q$ as an $\Lcal$-structure. It is infinite and it is uniquely listable by Corollary \ref{CoroQulist}. The relation $\ne$ is p.e. $\Lcal_a$-definable over $\Q$, hence, totally listable (cf. Corollary \ref{CoroTotList}). Thus, (ii) and (iii) are equivalent by Theorem \ref{ThmCharDPRM}. 

Assume (i). Then the identity map $\theta_1:\Z\to \Z$ defines a p.e. interpretation $\theta_1:\Q\dasharrow \Z$ as $\Lcal_a$-structures. We take $\theta_2:\Z\times (\Z-\{0\})\to \Q$ given by $(a,b)\mapsto a/b$, which defines a p.e. interpretation $\theta_2 : \Z\dasharrow \Q$; in fact, this is the p.e. interpretation $\kappa:\Z\dasharrow \Q$ from Lemma \ref{LemmaQinNZ}. For  $\zeta=\theta_2\bullet\theta_1:\Q\dasharrow \Q$ we have $\dom(\zeta)=\Z\times (\Z-\{0\})\subseteq \Q^2$ and
$$
\Gamma(\zeta)=\{(x_0,x_1,x_2)\in \Q^3 : x_1\in \Z, x_2\in \Z-\{0\}, x_0x_2=x_1\}\subseteq \Q^3
$$
which is p.e. $\Lcal_a$-definable over $\Q$ by (i). Hence, by Theorem \ref{ThmTransferDPRM} and Lemma \ref{LemmaZDPRM} we get (ii).

Finally, assume (ii). To obtain (i) it suffices to show that $\Z\subseteq \Q$ is $\gamma$-listable for some listable presentation $\gamma$ of $\Q$, because $\Q$ is uniquely listable. Directly doing this in detail without invoking Church's thesis can be rather messy if the listable presentation is not chosen carefully, but choosing the listable presentation $\tau:\N\to \Q$ provided by Proposition \ref{PropULQ} one can check that $\tau^{-1}(\Z)=\{2^n: n\ge 0\}\cup \{2^n-1: n\ge 0\}$, which is listable in $\N$.

Alternatively, consider again the p.e. interpretation $\kappa : \Z\dasharrow \Q$ from Lemma \ref{LemmaQinNZ}. Then $\kappa^*(\Z)=\{(a,b)\in \Z^2 : b\ne 0\mbox{ and }b|a\}$ is p.e. $\Lcal_a$-definable over $\Z$, hence totally listable (cf. Corollary \ref{CoroTotList}). Then $\Z$ is $\gamma$-listable in $\Q$ for some listable presentation $\gamma$ by Proposition \ref{PropIntimplList}.
\end{proof}

We stress the fact that, although $\N$ and $\Q$ are known to be bi-interpretable as $\Lcal_a$-structures thanks to results of J. Robinson \cite{JRobinsonQ}, it is not known whether they are p.e. bi-interpretable.

\begin{proposition}\label{PropOK} Let $K$ be a number field, let $O_K$ be its ring of integers, and let $\Gcal\subseteq O_K$ be a finite set of ring generators for $O_K$. Let us consider $O_K$ as a structure over $\Lcal=\Lcal_a\cup\Gcal$. Then $O_K$ is uniquely listable and the following are equivalent:
\begin{itemize}
\item[(i)] $\Z$ is Diophantine in $O_K$.
\item[(ii)]  The $\Lcal$-structure $O_K$ has the DPRM property.
\item[(iii)] The $\Lcal$-structure $O_K$ is p.e. bi-interpretable with the $\Lcal_a$-structure $\N$.
\end{itemize}
\end{proposition}
\begin{proof} We consider $\Z$ as an $\Lcal_a$-structure and $O_K$ as an $\Lcal$-structure.

The ring $O_K$ is infinite and it is uniquely listable by Corollary \ref{CoroULfg}. The relation $\ne$ is p.e. $\Lcal$-definable in $O_K$ ---this is elementary; see for instance paragraph 1.2.1 in \cite{MoretBailly}. Hence, (ii) and (iii) are equivalent by Theorem \ref{ThmCharDPRM}.

Let $r=[K:\Q]$ and let $\beta_1, ...,\beta_r$ be an integral basis for $O_K$ with $\beta_1=1$. The elements $\beta_j$ are p.e. $\Lcal$-definable since $\Gcal$ is a set of ring generators for $O_K$. Let $\kappa: \Z^r\to O_K$ be the map $\kappa(x_1,...,x_r)=x_1\beta_1+...+x_r\beta_r$. Thus, $\kappa$ defines a p.e. interpretation $\kappa: \Z\dasharrow O_K$.

Assume (i). Then the identity map $\theta_1:\Z\to \Z$ defines a p.e. interpretation $\theta_1:O_K\dasharrow \Z$. Let us take $\theta_2=\kappa : \Z\dasharrow O_K$. The composed interpretation $\zeta=\theta_2\bullet \theta_1:O_K\dasharrow O_K$ has $\dom(\zeta)=\Z^r\subseteq O_K$ and $
\Gamma(\zeta)=\{(x_0,x_1,...,x_r)\in O_K^{r+1} : x_1,...,x_r\in \Z\mbox{ and }x_0=x_1\beta_1+...+x_r\beta_r\}
$ which is p.e. $\Lcal$-definable by (i). We obtain (ii) by Theorem \ref{ThmTransferDPRM} and Lemma \ref{LemmaZDPRM}.

Finally, let us assume (ii). To get (i) it suffices to show that $\Z\subseteq O_K$ is listable for some listable presentation $\gamma$ of $O_K$ (since $O_K$ is uniquely listable). The p.e. interpretation $\kappa: \Z\dasharrow O_K$ constructed above satisfies $\kappa^{-1}(\Z)=\{(n,0,...,0): n\in \Z\}\subseteq \Z^r$ because $\beta_1=1$ and the $\beta_j$ form an integral basis. Thus, $\kappa^{-1}(\Z)$ is totally listable over $\Z$ (it is p.e. $\Lcal_a$-definable) hence $\Z$ is $\gamma$-listable over $O_K$ for some listable presentation $\gamma$ afforded by Proposition \ref{PropIntimplList}. 
\end{proof}
At this point we recall that it is a conjecture of Denef and Lipshitz \cite{DL78} that $\Z$ is Diophantine in $O_K$ for every number field. The general case remains open, although it is known that this would follow from standard conjectures on elliptic curves \cite{MazurRubin, MurtyPasten}.  The available unconditional results are proved in \cite{DL78,Denef80, Pheidas88, Shlapentokh89, Videla89, ShaShl89} and most recently in \cite{MRnew} and \cite{GFP}. All recent progress on this problem was possible thanks to the elliptic curve criteria from \cite{PoonenEll, CorPheZah, ShlapentokhEll}.

Regarding function fields, we have:
\begin{proposition}\label{Propkt} Let $k$ be a finite field, let $t$ be a transcendental element, and let $\Gcal$ be a (finite) set of ring generators for $k$. Let us consider $k[t]$ and $k(t)$ as $\Lcal$-structures, where $\Lcal=\Lcal_a\cup \Gcal \cup \{t\}$. The following are equivalent:
\begin{itemize}
\item[(i)] $k[t]$ is Diophantine in $k(t)$.
\item[(ii)] The $\Lcal$-structure $k(t)$ has the DPRM property.
\item[(iii)] The $\Lcal$-structures $k[t]$ and $k(t)$ are p.e. bi-interpretable.
\item[(iv)] The $\Lcal$-structure $k(t)$ is p.e. bi-interpretable with the $\Lcal_a$-structure $\N$.
\end{itemize}
\end{proposition}
\begin{proof}  The ring $k[t]$ is infinite, uniquely listable by Corollary \ref{CoroULfg}, and the inequality $\ne$ is p.e. $\Lcal$-definable (standard fact using two primes of $k[t]$; see Lemme 3.2 in \cite{MoretBailly} for a generalization). Furthermore, $k[t]$ has the DPRM property by a result of Demeyer \cite{DemeyerInv}. Thus $k[t]$ is p.e. bi-interpretable with the $\Lcal_a$-structure $\N$ by Theorem \ref{ThmCharDPRM}. It follows that (iii) and (iv) are equivalent by Lemma \ref{LemmaBiIntTransitive}.

The field $k(t)$ is infinite, uniquely listable by Corollary \ref{CoroULfg}, and the relation $\ne$ is p.e. $\Lcal$-definable in $k(t)$. Hence, (ii) and (iv) are equivalent by Theorem \ref{ThmCharDPRM}.

The equivalence of (i) and (ii) is shown as in the case of $\Z$ and $\Q$ (cf. Proposition \ref{PropZQ}) using the p.e. interpretation $\kappa: k[t]\dasharrow k(t)$ given by $\kappa: k[t]\times (k[t]-\{0\})\to k(t)$ with $(f,g)\mapsto f/g$. 
Namely, assuming (i), the identity map $\theta_1: k[t]\to k[t]$ defines a p.e. interpretation $\theta_1:k(t)\dasharrow k[t]$ and we can take $\theta_2=\kappa$ in order to apply Theorem \ref{ThmTransferDPRM} together with Demeyer's theorem \cite{DemeyerInv}. This allows us to transfer the DPRM property from $k[t]$ to $k(t)$, obtaining (ii).

Conversely, assume (ii). Then $\kappa^*(k[t])=\{(f,g)\in k[t]^2 : g\ne 0\mbox{ and }g|f\}$ is p.e. $\Lcal$-definable over $k[t]$, hence totally listable. Proposition \ref{PropIntimplList} implies that $k[t]$ is $\gamma$-listable in $k(t)$ for some listable presentation $\gamma$, hence $k[t]$ is a totally listable subset of $k(t)$ because $k(t)$ is uniquely listable. Therefore, (ii) implies (i). 
\end{proof}

Similarly, these results can be extended to the case of $S$-integers and global  fields, not just $\Q$ and $k(t)$. We leave the details to the reader.

In a similar fashion, other transference results can be obtained. For instance, Demeyer \cite{DemeyerPoly} proved that $\Q[t]$ has the DPRM property, seen as a structure over $\Lcal_t=\Lcal_a\cup\{t\}$. Using this and the same methods as in the previous three results, one can show

\begin{proposition} Consider $\Q[t]$ and $\Q(t)$ as structures over $\Lcal_t$. The following are equivalent:
\begin{itemize}
\item[(i)] $\Q[t]$ is Diophantine in $\Q(t)$.
\item[(ii)] The $\Lcal$-structure $\Q(t)$ has the DPRM property.
\item[(iii)] The $\Lcal$-structures $\Q[t]$ and $\Q(t)$ are p.e. bi-interpretable.
\item[(iv)] The $\Lcal$-structure $\Q(t)$ is p.e. bi-interpretable with the $\Lcal_a$-structure $\N$.
\end{itemize}
\end{proposition}


\section{Diophantine sets of global fields and related problems}\label{SecConjectures}

\subsection{Varieties and Diophantine sets} In this article, a variety over a field $k$ is a reduced separated scheme of finite type over $k$. In particular, we do not require irreducibility.

Recall that a set $S\subseteq k^n$ is Diophantine over $k$ if it is p.e. $\Lcal_a$-definable over $k$ with parameters from $k$. It easily follows from the definitions that $S$ is Diophantine over $k$ if and only if there is an affine variety $X$ over $k$ and a morphism $f:X\to \A^n_k$ defined over $k$ such that $f(X(k))=S$. 

\subsection{Mazur's conjecture} \label{SecMazur}

For later reference, let us recall some conjectures on the topology of rational points formulated by Mazur \cite{MazurConj1, MazurConj2}, as well as a variation proposed by Colliot-Th\'el\`ene, Skorobogatov, and Swinnerton-Dyer \cite{CTSSD}. 

The following intriguing conjecture is due to Mazur \cite{MazurConj1, MazurConj2}.
\begin{conjecture}[Mazur's conjecture]\label{ConjMazur} Let $X$ be a variety over $\Q$. The topological closure of $X(\Q)$ in $X(\R)$ has finitely many connected components. 
\end{conjecture}
 This conjecture was initially stated for smooth varieties, but the previous version is easily reduced to the smooth case by taking $X_1$ as the smooth locus of $X$, and then $X_2$ as the smooth locus of $X-X_1$, etc. which is a finite process because $X$ is of finite type over $\Q$.

As remarked by Mazur, this conjecture implies at once that $\Z$ is not Diophantine in $\Q$. By Proposition \ref{PropZQ}, it would also follow that $\Q$ does not have the DPRM property and that $\Q$ and $\N$ are not p.e. bi-interpretable as $\Lcal_a$-structures. 

Actually, a first version of Mazur's conjecture proposed in \cite{MazurConj1} asserted that the topological closure of $X(\Q)$ in $X(\R)$ precisely consisted of some of the connected components of $X(\R)$, but this was disproved in \cite{CTSSD}. Nevertheless, the following version of Mazur's conjecture proposed as Conjecture 4 in \cite{CTSSD} still seems plausible.
\begin{conjecture}[Strong version of Mazur's conjecture]\label{ConjStrongMazur} Let $X$ be a smooth irreducible variety over $\Q$ and let $U$ be a connected component of $X(\R)$. Let $W$ be the topological closure of $X(\Q)\cap U$ in $U$. Then there is a Zariski closed set $Y\subseteq X$ defined over $\Q$ such that $W$ is a finite union of connected components of $Y(\R)$. 
\end{conjecture}

The following observation is implicit in \cite{CTSSD} and in fact it motivates the previous conjecture.
\begin{lemma}\label{LemmaConjMazur} The strong version of Mazur's conjecture (Conjecture \ref{ConjStrongMazur}) implies Mazur's conjecture (Conjecture \ref{ConjMazur}). 
\end{lemma}
\begin{proof} The set of real points of an affine  variety defined over $\Q$ forms a semi-algebraic set, and it is a standard result that semi-algebraic sets over $\R$ have finitely many connected components.
\end{proof}

Let $L$ be a real-closed field and let $F\subseteq L$ be an ordered field. A semi-algebraic set $U\subseteq L^n$ is said to be \emph{defined over $F$} if there is a first order formula $\Phi[x_1,...,x_n]$ over the language $\Lcal_{\le}=\{0,1,+,\cdot, \le, =\}$ with parameters from $F$ such that $U$ is the interpretation of $\Phi$ over $L$. 

With this terminology, here is  yet another variant of Mazur's conjecture which is implicitly suggested in \cite{CTSSD}, and explicitly formulated at the end of Section 2 in \cite{CornelissenZahidi}. Here, for a variety $X$ over $\R$, a semi-algebraic set of $X(\R)$ is defined as a set which is semi-algebraic on each affine chart of an affine open cover of $X$.

\begin{conjecture}[Semi-algebraic version of Mazur's conjecture]\label{ConjSemiAlg} Let $X$ be a  variety over $\Q$. The topological closure of $X(\Q)$ in $X(\R)$ is a semi-algebraic set defined over $\Q$.
\end{conjecture}

It turns out that this last conjecture follows from Conjecture \ref{ConjStrongMazur} by the same argument as in Lemma \ref{LemmaConjMazur}, just keeping track of the field of definition.
\begin{lemma} The strong version of Mazur's conjecture (Conjecture \ref{ConjStrongMazur}) implies the semi-algebraic version of Mazur's conjecture (Conjecture \ref{ConjSemiAlg}).
\end{lemma}
\begin{proof} Let $Y$ be an affine algebraic variety defined over $\Q$ in the affine space $\A^n_\Q$. Let $C$ be a connected component of $Y(\R)$. It is a standard result that $C$ is semi-algebraic over $\R$ (cf. Theorem 5.22 in \cite{AlgRA}). We claim that the semi-algebraic set $C$ is defined over $\Q$ (this is well-known but we were not able to find an explicit reference for this particular fact). 

Let $F\subseteq \R$ be the field of real algebraic numbers. Then $F$ is real-closed, and by Proposition 5.24 in p.170 of \cite{AlgRA} there is a semi-algebraic connected component $C'$ of $Y(F)$ defined over $F$ such that for any $\Lcal_{\le}$-formula $\Phi[x_1,...,x_n]$ with parameters form $F$ which defines $C'$ over $F$,  the interpretation of $\Phi$ over $\R$ is $C$. Since $F$ is the real closure of $\Q$, Proposition 2.82 in p.71 of \cite{AlgRA} shows that there is an $\Lcal_{\le}$-formula $\Psi$ with parameters from $\Q$ (thus, $\Psi$ can be taken without parameters) such that the interpretation of $\Psi$ over $F$ is $C'$. Hence, the interpretation of $\Psi$ over $\R$ is $C$. This proves that the connected components of $Y(\R)$ are semi-algebraic sets defined over $\Q$.

Let $X$ be a smooth affine algebraic variety over $\Q$ in affine space $\A^n_\Q$; this case of Conjecture \ref{ConjSemiAlg} implies the general case by taking a suitable stratification of the variety under consideration (cf. the discussion after Conjecture \ref{ConjMazur}) and then taking affine coverings. Assuming Conjecture \ref{ConjStrongMazur} we get that the topological closure of $X(\Q)$ in $X(\R)\subseteq \R^n$ is the union of finitely many connected components of real algebraic sets $Y(\R)$ for certain varieties $Y$ defined over $\Q$. Thus, by the previous claim, the closure of $X(\Q)$ in $X(\R)$ is semi-algebraic defined over $\Q$.
\end{proof}

Let us observe the following:
\begin{proposition}\label{PropEndpoint} If the semi-algebraic version of Mazur's conjecture (Conjecture \ref{ConjSemiAlg}) holds, then for every Diophantine set $S\subseteq \Q$, we have that the topological closure of $S$ in $\R$ is a finite union of closed intervals whose endpoints are real algebraic or infinite. (Here, the singleton $\{x\}\subseteq \R$ is taken as the closed interval $[x,x]$.)
\end{proposition}
\begin{proof} This is by the Tarski-Seidenberg theorem with coefficients in $\Q$. See Theorem 2.77 in p.69 of \cite{AlgRA} for a general version with coefficients in an ordered field contained in a real closed field.
\end{proof}

One may ask about extensions of Mazur's conjecture to other global fields and other places, not just the archimedean. In the case of number fields, Mazur (cf. Question I in \cite{MazurConj2}) asked the following (see also \cite{PoonenShlapentokh}):

\begin{question}[Mazur]\label{QuestionMazur} Let $K$ be a number field and let $v$ be a place of $K$. Let $X$ be an irreducible projective variety over $K$. For each local point $x\in X(K_v)$, let $Z_x\subseteq X$ be the intersection of all Zariski closed sets $Y\subseteq X$ that contain some $X(K)\cap U$, as $U$ ranges over all $v$-neighborhoods of $x$. Is the collection $\{Z_x : x\in X(K_v) \}$ finite?
\end{question}

Let us remark the following:

\begin{lemma}\label{LemmaApplyMazur} Let $K$ be a number field, let $v$ be a place of $K$, and let $S\subseteq K$ be Diophantine over $K$. Assume either of the following:
\begin{itemize}
\item[(i)] $K=\Q$, $v=\infty$ is the archimedean place, and Mazur's Conjecture \ref{ConjMazur} holds; or
\item[(ii)] Mazur's Question \ref{QuestionMazur} has a positive answer for the number field $K$ and the place $v$. 
\end{itemize}
Then $S$ can have at most finitely many $v$-adically isolated points.
\end{lemma}
\begin{proof} Suppose that $S$ has infinitely many $v$-adically isolated points.  Let $X$  be an affine  variety defined over $K$ and let $f:X\to \A^1_K$ be a morphism defined over $K$ such that $S=f(X(K))$. Passing to a Diophantine subset of $S$ with infinitely many $v$-adically isolated points, we may assume that $X$ is irreducible. Let $z_1,z_2,...$ be an infinite sequence of points in $S$ that are $v$-adically isolated. Then the fibres $Y_j=f^{-1}(z_j)$ are infinitely many pairwise disjoint Zariski-closed subsets of $X$ defined over $K$ with $Y_j(K)$ non-empty. 

If $K=\Q$ and $v=\infty$, then the sets $Y_j(\R)$ for $j=1,2,...$ contain infinitely many connected components of the real closure of $X(\Q)$ in $X(\R)$ because $f$ maps them to isolated points of $f(X(\Q))$. Hence (i) cannot hold. In the general case, choose $x_j\in Y_j(K)$, let $V_j$ be a $v$-adic neighborhood of $z_j$ in $K_v$ that separates $z_j$ from $S-\{z_j\}$, and let $U_j=f^{-1}(V_j)$ which is a $v$-adic neighborhood of $x_j$. Then $Y_j$ contains $X(K)\cap U_j$, so $Z_{x_j}$ is contained in $Y_j$ in the notation of Question \ref{QuestionMazur}. As the varieties $Y_j$ are disjoint, (ii) cannot hold.
\end{proof}

Question \ref{QuestionMazur} is specific for number fields, and the analogue for global function fields is known to be false. Namely,  the following result essentially due to   Pheidas \cite{PheidasInv}  produces $v$-adically discrete infinite sets that are Diophantine in the function field setting (the connection with Mazur's conjecture was pointed out by Cornelissen and Zahidi \cite{CornelissenZahidi}).

\begin{theorem}\label{ThmDiscrete} Let $p>2$ be a prime. The sets $S_1 = \{t^{p^n} : n\ge 0\}$ and 
$$
S_2=\{b+t+t^p+t^{p^2}+...+t^{p^n} : n\ge 0\mbox{ and } b\in \F_p\}
$$ 
are Diophantine in $K=\F_p(t)$.  More precisely, let $U\subseteq \A^3_K$ be the curve defined over $\F_p(t)$ by 
$$
\begin{cases}
x-t=y^p-y\\
x^{-1} -t^{-1} = z^p-z.
\end{cases}
$$
Projecting $U(K)$ onto the $x$-coordinate gives $S_1$, and projecting onto the $y$-coordinate gives $S_2$.
\end{theorem}
\begin{proof} This is mostly contained in the proof of Lemma 1 of \cite{PheidasInv}. The only missing point is that from \emph{loc. cit.} one only gets $S_2\subseteq \pi(U(K))$ rather than equality, where $\pi:U\to \A^1_K$ is the projection onto the $y$-coordinate.

Let $A:\F_p(t)\to \F_p(t)$ be the map $A(f)=f^p-f$. Then $A$ is an additive group morphism with kernel $\F_p$. Let $(u,v,w)\in U(K)$ and note that $u=t^{p^n}$ for some $n\ge 0$ (cf. \cite{PheidasInv}). Then $A(v)=t^{p^n}-t$ and we easily check that $f=t+t^p+t^{p^2}+...+t^{p^{n-1}}$ satisfies
$$
A(f)=f^p-f=t^{p^n}-t.
$$
Hence, all the possibilities for $v$ are $A^{-1}(t^{p^n}-t)=\{b+f : b\in \F_p\}$.
\end{proof}

\subsection{Left-Diophantine numbers} \label{SecLeft}

For a real number $\alpha\in \R$, let $L(\alpha)=\{q\in \Q : q<\alpha\}$. We say that $\alpha\in \R$ is \emph{left-Diophantine} if $L(\alpha)$ is a Diophantine subset of $\Q$. (Naturally, there is a notion of right-Diophantine number but that leads to a similar analysis.)

By Lagrange's $4$-squares theorem, we have the elementary fact that the relation $\le$ is Diophantine over $\Q$. This will be frequently used in the discussion below. For instance, we deduce:
\begin{lemma}\label{LemmaSup} $\alpha\in \R$ is left-Diophantine if and only if there is a Diophantine subset $S\subseteq \Q$ such that $\alpha=\sup S$.
\end{lemma}
\begin{proof} If $L(\alpha)$ is Diophantine, then note that $\alpha=\sup S$ with $S=L(\alpha)$.  For the converse, if $S$ is Diophantine and $\alpha=\sup S$, we observe that $L(\alpha)=\{q\in \Q : \exists a\in S, q<a\}$, so $L(\alpha)$ is Diophantine.
\end{proof}

Let $\Dcal\subseteq \R$ be the collection of all left-Diophantine numbers and let $\Acal\subseteq \R$ be the set of real algebraic numbers. We have the following lower bound for $\Dcal$.

\begin{lemma} $\Acal\subseteq \Dcal$.
\end{lemma}
\begin{proof} Let $\alpha\in \R$ be algebraic with minimal polynomial $p(x)\in \Q[x]$. The roots of $p(x)$ are simple, so the function $p:\R\to \R$ changes sign at $\alpha$. Up to multiplying $p(x)$ by $-1$, we may assume that there are $q_1,q_2\in \Q$ such that $q_1<\alpha<q_2$ and for all $q_1\le u\le q_2$ we have $p(u)>0$ if $u<\alpha$ and $p(u)<0$ if $u>\alpha$. Then the set 
$$
X=\{u\in \Q : p(u)>0 \mbox{ and } u<q_2\}
$$
is Diophantine over $\Q$ and $\alpha=\sup X$. We conclude by Lemma \ref{LemmaSup}.
\end{proof}

Recall from Corollary \ref{CoroQulist} that the $\Lcal_a$-structure $\Q$ is uniquely listable. Thus, there is a well-defined notion of listable subsets of $\Q^r$ for every $r\ge 1$. 

A real number $\alpha\in \R$ is called \emph{left-listable} if $L(\alpha)\subseteq \Q$ is listable. This notion is standard and it appears under various names in the literature, such as left-r.e., and left-c.e. (see for instance Chapter 5 in \cite{DowneyHirschfeld}), although  the discussion on unique listability is omitted and replaced by the assumption that a \emph{recursive} presentation or ``effective coding'' for $\Q$ is fixed (see for instance Assumption 5.1.2 in \cite{DowneyHirschfeld}). The class $\Lambda$ of left-listable numbers is vast; for instance, $\Lambda$ contains all computable real numbers (those for which a Turing machine outputs the decimal expansion to any required precision) as well as more exotic real numbers such as Chaitin's constant. 

From Corollary \ref{CoroTotList} applied to the sets $L(\alpha)\subseteq \Q$ we deduce

\begin{lemma} $\Dcal\subseteq \Lambda$.
\end{lemma}

Thus, we know that $\Acal\subseteq \Dcal\subseteq \Lambda$. The set $\Lambda$ is much larger than $\Acal$ and one can ask for a better description of $\Dcal$. We expect the following:

\begin{conjecture}[Algebraicity]\label{ConjDA} All left-Diophantine numbers are algebraic. Thus,  $\Dcal=\Acal$.
\end{conjecture}

In particular, this would imply 

\begin{conjecture}\label{ConjDA2} $\Dcal$ is a field.
\end{conjecture}

In the direction of Conjecture \ref{ConjDA}, we have:

\begin{proposition} The algebraicity conjecture (Conjecture \ref{ConjDA}) follows from the semi-algebraic version of Mazur's conjecture (Conjecture \ref{ConjSemiAlg}). In particular, it follows from the strong version of Mazur's conjecture (Conjecture \ref{ConjStrongMazur}).
\end{proposition}
\begin{proof} This is by Proposition \ref{PropEndpoint}.
\end{proof}

In any case, the following much weaker conjecture seems plausible.
\begin{conjecture}\label{ConjwDA} Not every left-listable number is left-Diophantine. That is, $\Dcal\ne\Lambda$.
\end{conjecture}

The previous conjecture follows from the algebraicity conjecture (Conjecture \ref{ConjDA}). Furthermore:
\begin{proposition} Conjecture \ref{ConjDA2} implies  Conjecture  \ref{ConjwDA}. That is, if $\Dcal$ is a field, then $\Dcal\ne \Lambda$.
\end{proposition} 
\begin{proof} This is because $\Lambda$ is not even a ring; see \cite{Ambos} or Section 5.5 in \cite{DowneyHirschfeld}.
\end{proof}

It turns out that Conjecture \ref{ConjwDA} would be enough to show that $\Z$ is not Diophantine in $\Q$.

\begin{proposition} If $\Dcal\ne \Lambda$, then we have the following:
\begin{itemize}
\item[(i)] $\Z$ is not Diophantine in $\Q$.
\item[(ii)] The field $\Q$ does not have the DPRM property.
\item[(iii)] $\Q$ and $\N$ are not p.e. bi-interpretable as $\Lcal_a$-structures.
\end{itemize}
\end{proposition}
\begin{proof} If $\Dcal \ne \Lambda$ then there is some $\alpha\in \R$ such that $L(\alpha)\subseteq \Q$ is listable but it is not Diophantine, which implies that $\Q$ does not have the DPRM property. We conclude by Proposition \ref{PropZQ}.
\end{proof}

\subsection{Unboundedness of the positive existential rank} \label{SecKollar}

Let $\Mfrak$ be an $\Lcal$-structure with domain $M$. We say that $\Mfrak$ has  \emph{bounded p.e. rank} if there is a constant $B$ depending only on $\Mfrak$ such that for every p.e. $\Lcal$-definable set $S\subseteq M$ we have $\rk^{p.e.}_\Mfrak(S)\le B$. Otherwise, we say that $\Mfrak$ has \emph{unbounded p.e. rank}. (One can extend this notion to subsets of $M^r$, but the case $r=1$ is enough for our purposes.) For instance, Lemma \ref{LemmaCatBdd} shows that if $\Mfrak$ has a p.e. $r$-catalogue for some $r\ge 1$, then $\Mfrak$ has bounded p.e. rank. A well-known example:

\begin{lemma} $\N$ as an $\Lcal_a$-structure has bounded p.e. rank. In fact, it has p.e. $r$-catalogues for every $r\ge 1$. The same holds for the $\Lcal_a$-structure $\Z$.
\end{lemma}
\begin{proof} In the case of $\N$ this follows from the DPRM theorem, Theorem \ref{ThmCatDPRM}, and Lemma \ref{LemmaCatBdd}. For $\Z$ it is the same argument, using Lemma \ref{LemmaZDPRM}.
\end{proof}
For global fields, the author expects the following:

\begin{conjecture}\label{ConjBdd} Let $K$ be a global field and let $\Gcal$ be a finite set of field generators. Consider $K$ as a structure over $\Lcal=\Lcal_a\cup\Gcal$. Then $K$ has unbounded p.e. rank. In particular, it does not have p.e. $r$-catalogues for any $r\ge 1$.
\end{conjecture}

In support of this conjecture, we have the following

\begin{theorem}\label{ThmCt} Let us consider the field of rational functions $\C(t)$ as a structure over a language $\Lcal$ expanding $\Lcal_a\cup\{t\}$ by some symbols of constant. Then $\C(t)$ has unbounded p.e. rank. In particular, it does not have p.e. $r$-catalogues for any $r\ge 1$.
\end{theorem}
For the proof, we need a consequence of Koll\'ar's work \cite{Kollar} (Theorem \ref{ThmKollar} stated below), which we state only over $\C(t)$ for simplicity, although an analogous result holds over function fields of complex projective curves.

Given $n\ge 0$ let $\Rat_n$ be the set of all rational functions $\phi\in \C(t)$ of (topological) degree $n$ and note that $\C(t)=\cup_{n\ge 0} \Rat_n$. Writing such a $\phi$ as a fraction of polynomials and considering the coefficients of these polynomials (up to scaling, with the necessary non-vanishing conditions to ensure $\deg \phi=n$) we see that each $\Rat_n$ has a natural structure of quasi-projective variety. In particular, for a set $S\subseteq \C(t)$, it makes sense to consider varieties over $\C$ contained in $S$.

\begin{theorem}[Koll\'ar]\label{ThmKollar} Let $K=\C(t)$. Let $X$ be a variety over $K$ and let $f$ be a regular function on $X$ defined over $K$. Let $S=f(X(K))\subseteq K$. Let $d\ge 0$ be an integer such that $S$ contains a constructible set of dimension $d$ over $\C$. Then at least one of the following holds:
\begin{itemize}
\item[(i)] $\dim_K(X)\ge d$
\item[(ii)] For all but finitely many  $\alpha\in \C$, the set $S$ contains rational functions with poles at $\alpha$.
\end{itemize}
\end{theorem}
\begin{proof} This follows from Theorem 4 in \cite{Kollar}. Indeed,  (i) is implied by Theorem 4 (1) in \cite{Kollar} because $d$ is a lower bound for the Diophantine dimension of $S$ over $\C$ as defined there (as $S$ contains a constructible set of dimension $d$ over $\C$, it cannot be contained in a countable union of varieties of dimension smaller than $d$). On the other hand, each $\phi\in \C(t)$ of degree $n$ has an (effective) divisor of poles which in turn gives a point $\Pole(\phi)\in \Symm^n\Pro^1_\C(\C)$ where $\Symm^n\Pro^1_\C(\C)$ is the quotient of $\Pro^1_\C(\C)^n$ by the action of the symmetric group in $n$ letters. Let $\Pole_n(S)=\{\Pole(\phi) : \phi\in S\cap \Rat_n\}\subseteq \Symm^n\Pro^1_\C(\C)$. Theorem 4 (2) (b) in \cite{Kollar} asserts that for certain $n>0$ (in fact, infinitely many values of $n$) one has that $\Pole_n(S)$ contains a positive dimensional constructible set (namely, certain $\rho_m(D_m(\C))$ in the notation of \emph{loc. cit.}). Thus, the pre-image of $\Pole_n(S)\subseteq \Symm^n\Pro^1_\C(\C)$ in $\Pro^1_\C(\C)^n$ contains a positive dimensional constructible set, and so does some  (hence, each) coordinate projection. The constructible sets of $\Pro^1_\C(\C)$ are either finite or cofinite, hence item (ii) holds.
\end{proof}
With this at hand, we can prove Theorem \ref{ThmCt}.

\begin{proof}[Proof of Theorem \ref{ThmCt}] This argument builds on the same construction appearing in Example 6 (1) of \cite{Kollar}. Let $S_n\subseteq \C(t)$ be the set of polynomials of degree  $n$. Note that $\C$ is p.e. $\Lcal$-definable over $\C(t)$ thanks to the Riemann-Hurwitz formula; e.g. using the $\Lcal_a$-formula $\exists y, y^2=x^3+1$. Hence, $S_n=\{c_0+c_1t+...+c_nt^n: c_0,...,c_n\in \C\mbox{ and }c_n\ne 0\}\subseteq \Rat_n$ is p.e. $\Lcal$-definable over $\C(t)$. In particular, taking $r(n)=\rk^{p.e.}_{\C(t)}(S_n)$ we see that  there is an affine variety $X_n\subseteq \A^{r(n)+1}_{\C(t)}$ defined over $\C(t)$ such that 
$$
S_n = \{u_0\in \C(t) : \exists u_1...\exists u_{r(n)}, (u_0,u_1,...,u_{r(n)})\in X_n(\C(t))\}.
$$
Taking $f:X_n\to \A^1_{\C(t)}$ as the projection onto the first coordinate (a morphism defined over $\C(t)$) we see that $S_n=f(X_n(\C(t)))$.  Note that $S_n$ is isomorphic to $\C^{n}\times \C^\times$ as varieties over $\C$, so we can apply Theorem \ref{ThmKollar} with $d=n+1$. As (ii) does not hold, we must have $\dim_{\C(t)} X_n\ge d=n+1$. Since $X_n\subseteq \A^{r(n)+1}_{\C(t)}$, we conclude that $\rk^{p.e.}_{\C(t)}(S_n)=r(n)\ge n$.
\end{proof}

For our discussion, the relevant consequence of Conjecture \ref{ConjBdd} is the following.

\begin{proposition}\label{PropBddnonDioph}  Let $K$ be a global field and let $\Gcal$ be a finite set of field generators for it. Consider $K$ as a structure over  $\Lcal=\Lcal_a\cup\Gcal$. If Conjecture \ref{ConjBdd} holds for $K$, then $K$ does not have the DPRM property and it is not p.e. bi-interpretable with the $\Lcal_a$-structure $\N$. Thus:
\begin{itemize}
\item[(i)] In the case $K=\Q$ this implies that $\Z$ is not Diophantine in $\Q$.
\item[(ii)] In the case $K=k(t)$ for a finite field $k$, this implies that $k[t]$ is not Diophantine in $k(t)$ and that the field $k(t)$ is not p.e. bi-interpretable with $k[t]$.
\end{itemize}
\end{proposition}
\begin{proof} The $\Lcal$-structure $K$ is uniquely listable by Corollary \ref{CoroULfg}. Hence, the first part follows from Theorem \ref{ThmCatDPRM}, Lemma \ref{LemmaCatBdd}, and Theorem \ref{ThmCharDPRM}. 

In addition, (i) follows from Proposition \ref{PropZQ} while (ii) follows from Proposition \ref{Propkt}.
\end{proof}
We remark that the analogue of items (i) and (ii) in Proposition \ref{PropBddnonDioph} for a number field $K$ is conditional to $\Z$ being Diophantine in $O_K$, which is not known in general. See Proposition \ref{PropOK} and the discussion after it.

For more information about the p.e. rank in the case of fields, we refer the reader to  \cite{DDF} by Daans, Dittmann, and Fehm. In particular, they independently arrived to the observation that if the p.e. rank of Diophantine subsets of $\Q$ is unbounded, then $\Z$ is not Diophantine in $\Q$.

%
\subsection{A Diophantine approximation conjecture}\label{SecKey} Let $K$ be a global field and let $v$ be a place of $K$ with normalized absolute value $|-|_v$; that is, if $v$ corresponds to a prime ideal $\pfrak$ then $|x|_v=[O_K:\pfrak]^{\ord_\pfrak(x)}$ while if $v$ corresponds to a (possibly real) embedding $\sigma: K\to \C$ then $|x|_v=|\sigma(x)|$ where $|-|$ is the usual absolute value on $\C$. Let $K_v$ be the completion of $K$ at $v$. For every projective variety $X$ over $K$ and every Cartier divisor $D$ on $X$ defined over $K$, there is a local Weil function $\lambda_{X,D,v}:X(K_v)-\supp(D)\to \R$, where $\supp(D)$ denotes the support of $D$; see Section 6.2 in \cite{Serre} for a precise definition and construction. Roughly speaking, if $D$ is represented by $\{(U_j, f_j)\}_j$ for some finite open cover $\{U_j\}_j$ of $X$ and $f_j$ are the corresponding local equations of $D$, then  $\lambda_{X,D,v}(P)= -\log |f_{D,j}(P)|_v + \alpha_j(P)$ for all $P\in U_j-\supp(D)$, where $\alpha_j$ is a certain nice bounded function. In particular, if $D$ is effective and $P_1,P_2,...$ is a sequence in $X(K_v)-\supp(D)$, then $\lambda_{X,D,v}(P_j)\to \infty$ if and only if the sequence $P_j$ approaches $\supp(D)(K_v)$ $v$-adically. 

Since we will be concerned with points in the support of divisors, it is worth pointing out the following clarification. If $X$ is a variety over a field $k$ and $D$ is an effective Cartier divisor on $X$ defined over $k$, then $D$ is locally given by equations over $k$ on an open covering of $X$, and there is no need for the variety $\supp(D)$ to have $k$-rational points. For instance, if $k=\Q$ and $X=\Pro^1_\Q = {\rm Proj}\, \Q[x_0,x_1]$, we have the open covering defined over $\Q$ given by $U_0=\{[x_0:x_1]:x_0\ne 0\}$ and $U_1=\{[x_0:x_1]:2x_0^2\ne x_1^2\}$. Let us define the divisor $D$ represented by $\{(U_0, 2x_0^2-x_1^2), (U_1,1) \}$. Then $D$ is an effective divisor of degree $2$ defined over $\Q$ with $\supp(D)(\Q)=\emptyset$ and $\supp(D)(\R)=\{[1:-\sqrt{2}],[1:\sqrt{2}]\}$.

We would like to propose the following:

\begin{conjecture}\label{ConjKey} Let $K$ be a global field and let $v$ be a place of $K$. Let $X$ and $Y$ be positive dimensional irreducible projective varieties over $K$, let $f : X\dasharrow Y$ be a dominant rational map defined over $K$, and let $U$ be a non-empty Zariski open set of $X$ defined over $K$ and contained in the domain of $f$. Suppose that $X(K)$ is Zariski dense in $X$. Then there is an effective Cartier divisor $D$ on $Y$ defined over $K$, such that $\lambda_{Y, D, v}$ is unbounded on $f(U(K))-\supp(D)$. That is, there is a sequence of $K$-rational points in  $U-f^{-1}(\supp (D))$ whose images under $f$ approach $\supp(D)(K_v)$ $v$-adically.
\end{conjecture}

Applications of this conjecture in the study of Diophantine subsets of global fields will be discussed in Section \ref{SecnD}. A slightly different version of Conjecture \ref{ConjKey} using morphisms instead of rational maps was proposed by the author in \cite{OWR}, but the current form of the conjecture is simpler to use in applications and it gives a wider range of search for potential counterexamples (if any).

Even the simplest case when $U=X=Y$ and $f=\Id_X$ is open in general. In this case, Conjecture \ref{ConjKey} says that there is a sequence of $K$-rational points of $X$ that $v$-adically accumulates towards the support of some effective Cartier divisor on $X$ defined over $K$.

Conjecture \ref{ConjKey} can be reduced to a special case.

\begin{proposition}\label{PropP1} Let $K$ be a global field and $v$ a place of $K$. If Conjecture \ref{ConjKey} holds for $K$ and $v$ in the special case when $Y=\Pro^1_K$, then it holds in general for this choice of $K$ and $v$.
\end{proposition}
\begin{proof} In the setup of Conjecture \ref{ConjKey}, consider a non-constant rational function $g:Y\dasharrow \Pro^1_K$ and let $Z\subseteq Y$ be the locus where $g$ is not defined. Let $U\subseteq X$ be a non-empty Zariski open set defined over $K$ and contained in $\dom(f)$. Let $U'=U-f^{-1}(Z)$ and note that $U'$ is contained in $\dom(g\circ f)$. Let us apply the conjecture to $g\circ f:X\dasharrow \Pro^1_K$ and $U'$. This gives a divisor $D$ on $\Pro^1_K$. Let $E$ be any effective Cartier divisor on $Y$ defined over $K$ whose support contains $g^{-1}(\supp(D))\cup Z$.

Let $x_1,x_2,...$ be a sequence in $U'(K)$ (hence, in $U(K)$)  such that $(g(f(x_j)))_{j\ge 1}$ $v$-adically approaches $\supp(D)(K_v)$.  Since $Y$ is projective and $K_v$ is locally compact, $Y(K_v)$ is compact and the sequence $(f(x_j))_{j\ge 1}$ has a $v$-adic accumulation point $y\in Y(K_v)$. Note that either $y\in Z(K_v)$ or $g$ is defined at $y$, in which case $y \in g^{-1}(\supp(D))(K_v)$. In either case, $y\in \supp(E)(K_v)$. Therefore, $(f(x_j))_{j\ge 1}$ has a subsequence that $v$-adically approaches $\supp(E)(K_v)$.
\end{proof}

The case $K=\Q$ is particularly relevant for us since, as we will see, Conjecture \ref{ConjKey} for $K=\Q$ implies that $\Z$ is not Diophantine in $\Q$ (cf. Proposition \ref{PropKeyQ} below). In this case we have

\begin{proposition} Mazur's Conjecture \ref{ConjMazur} implies Conjecture \ref{ConjKey} for $K=\Q$ and $v=\infty$ the archimedean place.  Furthermore, a positive answer to Mazur's question \ref{QuestionMazur} for a given number field $K$ and a place $v$ of it implies Conjecture \ref{ConjKey} for this choice of $K$ and $v$.
\end{proposition}
\begin{proof} By Proposition \ref{PropP1} it suffices to consider the case of an irreducible projective variety $X$ over $K$ with $X(K)$ Zariski dense and a dominant rational function $f:X\dasharrow \Pro^1_K$ defined over $K$. Let $U\subseteq X$ be a non-empty Zariski open set defined over $K$. Shrinking $U$ if necessary we may assume that $U$ is affine and that $f:U\to \A^1_K$ is a regular function defined over $K$. Then $f(U(K))\subseteq K$ is an infinite Diophantine set and in either case (assuming Conjecture \ref{ConjMazur} or a positive answer to Question \ref{QuestionMazur}) Lemma \ref{LemmaApplyMazur} would give that $S=f(U(K))$ has at most finitely many $v$-adically isolated points. Let $y\in S$ be a point which is not $v$-adically isolated in $S$. Then we can take $D$ as the divisor on $\Pro^1_K$ determined by $y\in f(U(K))$, which is as required by Conjecture \ref{ConjKey} because there is a sequence in $S-\{y\}$ that $v$-adically converges to $y$.
\end{proof}

In particular, the available evidence for Mazur's Conjecture \ref{ConjMazur} and for a positive answer to Question \ref{QuestionMazur}  provides evidence for Conjecture \ref{ConjKey} in the number field setting.

Since the analogue of Mazur's conjecture fails over function fields due to the example provided by Theorem \ref{ThmDiscrete} and since the previous result shows a close connection between Mazur's conjecture and Conjecture \ref{ConjKey}, one may ask whether Theorem \ref{ThmDiscrete} can also be used to give a counterexample to the function field version of Conjecture \ref{ConjKey}. On the contrary, it gives a rather non-trivial example where Conjecture \ref{ConjKey} has essentially one chance to work, and it does.

\begin{example} For $p>2$ let $K=\F_p(t)$ and let $v$ be the $t$-adic valuation on $K$. Let $U$ be the afine curve in Theorem \ref{ThmDiscrete} and let $f, g:U\to \A^1_K\subseteq \Pro^1_K$  be the projection maps from $U$ onto the $x$ and $y$ coordinates respectively. Let $X$ be a projective closure of $U$ and extend $f,g$ to rational functions $f,g:X\dasharrow \Pro^1_K$.

Let $T$ be the affine coordinate on $\A^1_K$. For the map $f$ we note that the only $t$-adic limit point of $f(U(K))=S_1=\{t^{p^n}: n\ge 0\}$ in $K_v=\F_p((t))$ is $0$. So, we can take the divisor $D=\{T=0\}$. 

On the other hand, the elements of $g(U(K))=S_2$ $t$-adically accumulate towards the formal power series $f_b=b+t+t^p+t^{p^2}+...\in \F_p[[t]]\subseteq K_v$ for $b\in \F_p$, and these are the only limit points. The power series $f_b$ for $b\in \F_p$ are algebraic over $\F_p(t)$ and, in fact, they are the roots of $T^p-T+t$. We can take the divisor $D=\{T^p-T+t=0\}$ on $\A^1_K\subseteq \Pro^1_K$.

In both cases, Conjecture \ref{ConjKey} holds with the given choices of the effective divisor $D$ on $\A^1_K\subseteq \Pro^1_K$. We note that since the Diophantine sets $S_1$ and $S_2$ have only finitely many $t$-adic limit points, the choice of $D$ is essentially unique (up to multiplicity and up to adding more components).
\end{example}

\subsection{Consequences of the Diophantine approximation conjecture} \label{SecnD}

Applications of Conjecture \ref{ConjKey} are simplified by the following observation.

\begin{lemma}\label{LemmaKey} Let $K$ be a global field and let $v$ be a place of $K$. Let $S\subseteq K$ be an infinite Diophantine subset which is $v$-adically bounded in $K$ and has exactly one $v$-adic limit point  $\alpha\in K_v$. If Conjecture \ref{ConjKey} holds for $K$ and $v$, then $\alpha$ is algebraic over $K$.
\end{lemma}
\begin{proof} Since $S$ is Diophantine, there is an affine variety $U$ over $K$ and a morphism $f:U\to \A^1_K$ such that $f(U(K))=S$. Possibly passing to a subsequence in $S$ and shrinking $U$, we may assume that $U$ is irreducible and that $U(K)$ is Zariski dense in $U$. Let $X$ be a projective closure of $U$ and extend $f$ to a rational function $f: X\dasharrow \Pro^1_K$.

Since $S$ is $v$-adically bounded in $K$, the point $\alpha\in K_v\subseteq \Pro^1_K(K_v)$ is the only $v$-adic limit point of $f(U(K))=S$ in $\Pro^1_K(K_v)$. If $\alpha$ were transcendental over $K$, it would not belong to the support of any Cartier divisor on $\Pro^1_K$ defined over $K$. This would contradict Conjecture \ref{ConjKey}.
\end{proof}
\begin{proposition} If Conjecture \ref{ConjKey} holds for a global field $K$ and any place $v$, then $K$ does not have the DPRM property over the language $\Lcal=\Lcal_a\cup\Gcal$, where $\Gcal$ is a finite set of field generators for $K$. In particular, $K$ is not p.e. bi-interpretable with the $\Lcal_a$-structure $\N$.
\end{proposition}
\begin{proof} One can take a listable set $S$ consisting of the terms of a sequence that $v$-adically converges to a transcendental $\alpha\in K_v$ ---such examples are easy to produce using, for instance, the same idea as in Liouville's explicit construction of transcendental numbers. The result now follows from Lemma \ref{LemmaKey}. The final part is due to Theorem \ref{ThmCharDPRM}.
\end{proof}

From Propositions \ref{PropZQ} and \ref{Propkt} we immediately deduce:
\begin{proposition}\label{PropKeyQ} If Conjecture \ref{ConjKey} holds for $\Q$ and any $v$, then $\Z$ is not Diophantine in $\Q$. 
\end{proposition}

\begin{proposition} Let $k$ be a finite field. If Conjecture \ref{ConjKey} holds for $K=k(t)$ and any place $v$, then $k[t]$ is not Diophantine in $k(t)$. 
\end{proposition}

The non-Diophantineness of $\Z$ in $\Q$ and of $k[t]$ in $k(t)$ is also implied by different conjectures, as  discussed in previous sections. However, Conjecture \ref{ConjKey} has other consequences:

\begin{proposition}\label{Proptn} Let $k$ be a finite field. If Conjecture \ref{ConjKey} holds for $K=k(t)$ and the place $v$ given by the $t$-adic valuation, then $\{t^n: n\ge 0\}$ is not Diophantine in $k(t)$. 
\end{proposition}

This should be compared to the following result of Pheidas (see Lemmas 1 and 3 in \cite{PheidasInv} for characteristic $p\ge 3$, and see \cite{Videla}  for characteristic $2$; see also \cite{PastenFrob} for a generalization to function fields of bounded genus uniformly on the characteristic):

\begin{theorem}[Pheidas]\label{ThmPheidasps} Let $k$ be a finite field of characteristic $p>0$. The binary relation $x\le_p y $ on $k(t)$ defined by $\exists s\ge 0, y=x^{p^s}$ is Diophantine over $k(t)$. In particular, the set $\{t^{p^n} : n\ge 0\}$ is Diophantine  in $k(t)$.
\end{theorem}

Let us recall some preliminaries for the proof of Proposition \ref{Proptn}.

Let  $m\ge 2$ be an integer. By \emph{$m$-automaton}, we mean a finite deterministic automaton with input language $\{0,1,...,m-1\}$ and output states $0$ (``reject'') and $1$ (``accept''). We say that a set $A\subseteq \N$ is \emph{$p$-recognizable} if there is a $p$-automaton $\Mcal_A$ that computes the characteristic function of $A$. More precisely, for an integer $r\in \N$,  the $m$-automaton $\Mcal_A$ takes as input the string formed by the base $m$ digits of $r$ (from less to most significant) and it reaches the final state $0$ or $1$ according to $r\not\in A$ or $r\in A$ respectively.

For a set $A\subseteq \N$ we define the counting function $N(A,x)=\#\{n\in A : n\le x \}$. Let us recall the following classical estimate.

\begin{lemma}\label{LemmaRec} Let $m\ge 2$ and let $A\subseteq \N$ be $m$-recognizable. If $A$ is infinite, then there are constants $x_0\ge 1$ and  $c>0$ such that for all $x\ge x_0$ we have $N(A,x)>c\cdot \log x$.
\end{lemma}
\begin{proof} See Proposition 5 in \cite{MinskyPapert} for the case $m=2$; the general case is proved in the same way. Alternatively, there is the stronger estimate provided by Theorem 12 in \cite{Cobham}.
\end{proof}

Let $A\subseteq \N$ and let $p$ be a prime. The generating series of $A$ over $\F_p$ is the formal power series $f_A=\sum_{a\in A} t^a\in \F_p[[t]]$.  The following result is due to Christol \cite{Christol}

\begin{theorem}[Christol]\label{ThmChristol} Let $A\subseteq \N$ and let $p$ be a prime. The set $A$ is $p$-recognizable if and only if $f_A\in \F_p[[t]]$ is algebraic over $\F_p(t)$.
\end{theorem}

For a prime $p$, the binary relation $|_p$ on $\N$ is defined as follows: $x|_py$ if and only if there is $s\ge 0$ with $y=xp^s$. The following definability result is due to Pheidas \cite{Pheidasdivp}.

\begin{theorem}[Pheidas]\label{ThmPheidasdivp} The multiplication function $\cdot: \N\times \N\to \N$ is p.e. definable over the structure $(\N; 0,1,+,|_p,=)$. 
\end{theorem}
With these results at hand, we can prove Proposition \ref{Proptn}.

\begin{proof}[Proof of Proposition \ref{Proptn}] Assume Conjecture \ref{ConjKey} for $k(t)$ and $v$ the $t$-adic valuation. Let $p$ be the characteristic of $k$. The rule $t\mapsto t+1$ defines a $k$-linear field automorphism of $k(t)$, so it suffices to show that $P=\{(1+t)^n : n\ge 0\}$ is not Diophantine over $k(t)$.

Let $A\subseteq \N$ be the set of integers which can be written as the sum of different integers of the form $p^{j^j}$ for $j\ge 1$. That is, $A$ consists of all integers whose base $p$ expansion only has digits $0$ and $1$, and the digit $1$ can only occur in front of $p^n$ for $n=1,4,27,...,j^j,...$ Explicitly,
$$
A=\{0, p, p^4, p+p^4, p^{27},p+p^{27}, p^4+p^{27}, p+p^4+p^{27},...\}
$$

We note that for $x=p^{j^j}$ we have $N(A,x)\le 2^{j}$, as every $n\in A$ with $n\le x$ is uniquely determined by some subset of $\{p^{i^i}: 1\le i\le j\}$ ---in fact, $N(A,x)=1+2^{j-1}$. Thus, for every $c>0$ there is some $x_0$ such that for all $x=p^{j^j}>x_0$ we have $N(A,x)<   j^j\cdot c\log p=c\log x$. By Lemma \ref{LemmaRec} and Theorem \ref{ThmChristol}, the power series $f_A\in \F_p[[t]]\subseteq k[[t]]$ is transcendental over $k(t)$.

For $r\ge 1$ let $n_r=\sum_{j=1}^r p^{j^j}$ and define $S=\{n_r : r\ge 1\}=\{p, p+p^4, p+p^4+p^{27},...\}\subseteq \N$. Then 
$$
(1+t)^{n_r} = \prod_{j=1}^r (1+t)^{p^{j^j}}=\prod_{j=1}^r \left(1+t^{p^{j^j}}\right).
$$
From the last product, the sequence of polynomials $(1+t)^{n_r}$ converges $t$-adically to $f_A$ as $r\to \infty$.

For the sake of contradiction, suppose that $P=\{(1+t)^{n} : n\ge 0\}$ is Diophantine over $k(t)$. Then, by Theorem \ref{ThmPheidasps}, the map $\theta: P\to \N$ given by $\theta((1+t)^n)=n$ determines a p.e. interpretation of the structure $(\N;0,1,+,|_p,=)$ in $k(t)$, where $k(t)$ is seen as a structure over $\Lcal=\Lcal_a\cup\Gcal$ for a finite set $\Gcal$ of field generators for $k(t)$. By Theorem \ref{ThmPheidasdivp}, $\theta$ also determines a p.e. interpretation of $\N$ seen as an $\Lcal_a$-structure. Note that the set $S=\{n_r : r\ge 1\}\subseteq \N$ defined above is listable because it is defined from elementary arithmetic functions. By the DPRM theorem (on $\N$) we get that $S$ is p.e. $\Lcal_a$-definable over $\N$. Hence, $\theta^*(S)=\{(1+t)^{n_r} : r\ge 0\}$ is p.e. $\Lcal$-definable over $k(t)$ because $\theta$ is a p.e. interpretation. In particular, $T:=\{(1+t)^{n_r} : r\ge 0\}$ is Diophantine in $k(t)$.

The Diophantine set $T$ is $t$-adically bounded and we proved that its only $t$-adic limit point is $f_A\in k[[t]]$, which is transcendental over $k(t)$. By Lemma \ref{LemmaKey}, this contradicts Conjecture \ref{ConjKey}.
\end{proof}

We remark that, alternatively, the use of Christol's theorem in the previous argument can be replaced by an ad hoc transcendence lemma for lacunary power series, after some modification of the construction. However, we feel that Christol's theorem might be relevant in approaching Conjecture \ref{ConjKey} in the function field setting, which is why we decided to highlight this connection.

\subsection{Final questions} 
\begin{question} Does every finitely generated infinite domain have the DPRM property? In other words, is every finitely generated infinite domain p.e. bi-interpretable with the semi-ring $\N$?
 \end{question}
\begin{question} Is there some finitely generated  field which has the DPRM property? Namely, is there some finitely generated  field which is p.e. bi-interpretable with the semi-ring $\N$?
 \end{question}


%
\section{Acknowledgments}

I would like to thank Barry Mazur and Bjorn Poonen for several discussions  on the conjectures presented in Section \ref{SecConjectures}. I am also grateful to Philip Dittmann for his comments on a first version of this manuscript and for informing me about his joint work with Nicolas Daans and Arno Fehm \cite{DDF}. The extremely valuable comments by the referee are gratefully acknowledged. This research was supported by ANID (ex CONICYT)  FONDECYT Regular grant 1190442 from Chile.


\end{document}